\documentclass[article]{siamart250211}

\usepackage{amssymb}
\usepackage{comment}
\newsiamremark{assumption}{Assumption}
\newsiamremark{remark}{Remark}

\hypersetup{
  pdftitle={On quantitative sufficient second-order optimality conditions for elliptic optimal control problems
  },
  pdfauthor={Francisco Fuica and Nicolai Jork}
}

\title{On quantitative sufficient second-order optimality conditions for elliptic optimal control problems \thanks{\funding{The first author was partially supported by ANID through FONDECYT grant 11260142. This work was funded by the Deutsche Forschungsgemeinschaft (DFG, German Research Foundation), project number 577735529.}}}
\author{
Francisco Fuica \thanks{Departamento de Matem\'atica y Ciencia de la Computaci\'on, Universidad de Santiago de Chile, Santiago, Chile (\email{francisco.fuica@usach.cl}).}
\and Nicolai Jork \thanks{Department of Mathematics, University of T\"ubingen, 72076 T\"ubingen, Germany
(\email{nicolai.jork@uni-tuebingen.de}).}}
\headers
{Quantitative second-order optimality conditions}
{F. Fuica and N. Jork}
\begin{document}

\maketitle
\begin{abstract}
In this paper, a general principle for coercivity in optimal control problems is investigated. Utilizing regularity properties of solutions to second-order elliptic PDEs on Lipschitz domains, we provide quantitative sufficient second-order optimality condition for control-constrained semilinear elliptic optimal control problems. The main property of the investigated optimal control problems is the absence of a Tikhonov regularization term on the objective function. For these problems, the second variation is a quadratic form in the linearized state, and its curvature function may vanish or change sign. The quantitative condition we introduce implies coercivity with respect to the $L^2$-norm of the linearized-states. As a consequence, the stability of the sufficient second-order condition under perturbations of the states and the tracking data is shown.
\end{abstract}

\section{Introduction}
\label{sec1}
Sufficient second-order optimality conditions are fundamental for the analysis of smooth nonlinear and nonconvex optimization problems. For these problems, first-order conditions do not in general guarantee local optimality. This applies to PDE-constrained optimization and optimal control, which are the focus of the present paper. Classical second-order sufficient conditions for such problems are formulated as coercivity properties of the second variation on a suitable critical cone. See, for example, Casas and Casas-Mateos \cite{Casas2012, CasasMateos2020} and the references therein. Besides implying local optimality, conditions of this type play an important role in stability analysis and in the analysis of numerical approximations.

A major difficulty is that such coercivity conditions are often hard to verify at a given stationary control. Several results on sufficient second order conditions and approaches for the verification of these sufficient second-order conditions have been considered in the literature. In the finite-dimensional setting of control-affine ordinary differential equations (ODEs), in \cite{AgrachevStefaniZezza2002,MaurerOsmolovskii2004} it was shown that second variations associated with switching-time perturbations provide sufficient conditions for strong local optimality of bang-bang extremals under suitable regularity assumptions. For a systematic treatment of optimality conditions for regular and bang-bang optimal control problems governed by finite-dimensional ODEs, we refer the reader to \cite{OsmolovskiiMaurer2012}. In the infinite-dimensional setting, the numerical verification of sufficient optimality conditions for PDE-constrained optimization problems regularized by a Tikhonov term was investigated by Rösch and Wachsmuth \cite{zbMATH05609474}. In contrast, in the unregularized setting considered here, the second variation contains no positive quadratic control term, which makes the verification of second-order coercivity delicate. Optimal controls in the unregularized case are typically of bang-bang type or contain a bang-bang part, see \cite{zbMATH07865499} and the included references. For additional results concerning the growth of nonconvex functionals at strict local minimizers in optimal control, we refer to \cite{zbMATH08209888}. 

The purpose of this paper, motivated by the above results, is to provide quantitative conditions on a given stationary point, which imply the classical second-order sufficient condition for unregularized problems. The main idea is that a possible loss of pointwise positivity of the coefficient function describing the curvature of the objective functional, can be compensated by regularity properties of the corresponding linearized states. This allows the curvature coefficient to even be negative on a part of the domain. The conditions investigated here are different from structural assumptions on adjoint level sets that are used to derive first-order growth for bang-bang controls. Such assumptions and their role in second-order analysis have been studied in \cite{Casas2012,CasasWachsmuth2017,CasasMateos2020}. In the present paper, the geometry of adjoint level sets is used instead to control the spatial distribution of the linearized state and thus to establish coercivity of the second variation. Related
second-order conditions for bang-bang and bilinear control problems are discussed in \cite{CasasWachsmuthBilinear2018,zbMATH06913572,arXiv:2602.14632}. Second-order conditions of related type have previously been utilized
in the stability analysis of optimal states under perturbations in \cite{zbMATH07695675}.

To present our approach, we consider a semilinear elliptic problem with homogeneous Dirichlet boundary conditions. Extensions to other PDE- and ODE-constrained problems, different state equations, and other boundary conditions are left for future work.

To describe the problem and the main results, we let $\Omega\subset\mathbb R^n$, $n\in\{2,3\}$, be a bounded domain and introduce the set of feasible controls
\[
\mathcal U:=\left\{u\in L^2(\Omega):\alpha\le u(x)\le\beta\quad\text{for a.e. }x\in\Omega\right\},\quad\alpha<\beta.
\]
For a desired state $y_d$, we consider the optimal control problem
\[
\min_{u\in\mathcal U}
\left\{J(u):=\frac12\int_\Omega(y_u-y_d)^2\,dx\right\},
\]
where $y_u$ solves
\begin{equation}
\label{EQ1.1}
-\operatorname{div}(A\nabla y_u)+f(\cdot,y_u)=u\quad\text{in }\Omega,\quad y_u=0\quad\text{on }\partial\Omega.
\end{equation}
Assumptions on $A$ and $f$ are specified below in Assumption \ref{A2.1}.
Let $\bar u\in\mathcal U$ be a stationary point, by which we mean that $\bar u$ satisfies
\[
J'(\bar u)(u-\bar u)\ge0\quad\forall u\in\mathcal U.
\]
The corresponding state and adjoint state are denoted by
$\bar y$ and $\bar p$, respectively. For a direction $h$, let
$z_{\bar u,h}$ be the corresponding linearized state. The second
variation is given for $\bar c:=1-\bar p f_{yy}(\cdot,\bar y)$ by
\[
J''(\bar u)[h,h]=\int_\Omega\bar c(x)z_{\bar u,h}(x)^2\,dx.
\]
The curvature coefficient $\bar c$ appearing in the objective functional $J$ does not need to be positive. It may vanish on large sets or it may change its sign. Still, the standard second-order theory for unregularized problems relies on coercivity with respect to the linearized state. For $\tau>0$, we consider the extended critical cone
\[
C_{\bar u}^{\tau}:=\left\{h\in T_{\mathcal U}(\bar u):h=0\quad\text{a.e. on }[|\bar p|>\tau]\right\}.
\]
A classical sufficient condition is the existence of
$\tau,\kappa>0$ such that
\begin{equation}
\label{EQ1.2}
J''(\bar u)[h,h]\ge\kappa\|z_{\bar u,h}\|_{L^2(\Omega)}^2\quad\forall h\in C_{\bar u}^{\tau}.
\end{equation}
Such conditions are well established and we refer the reader to \cite{Casas2012,CasasTroeltzsch2015,CasasMateos2020} for the classical results regarding the conditions. The central question addressed in this paper is the following:
\medskip

\emph{What quantitative properties of a stationary point are sufficient to guarantee the state-coercivity \eqref{EQ1.2}?}

\medskip

To motivate the approach in our paper, suppose that \eqref{EQ1.2} is satisfied with some $\kappa>0$. Then for $h\in C_{\bar u}^{\tau}$,
\[
\int_{[\bar c<\kappa]}(\kappa-\bar c)z_{\bar u,h}^2\,dx\le\int_{[\bar c>\kappa]}(\bar c-\kappa)z_{\bar u,h}^2\,dx.
\]
From this, it is clear that if the curvature is small or negative on a subset of the domain, state coercivity of the second variation can hold only if the linearized states belonging to the extended critical directions do not concentrate on the unfavourable set. Let us assume that $\|f_{yy}(\cdot,\bar y)\|_{L^\infty(\Omega)}>0$. For $0<k<\frac1{\|f_{yy}(\cdot,\bar y)\|_{L^\infty(\Omega)}}$, we define
$\delta(k):=1-\|f_{yy}(\cdot,\bar y)\|_{L^\infty(\Omega)}k>0$ and $M(k):=\|\bar c^-\|_{L^\infty(\{|\bar p|\ge k\})}$, $\bar c^-:=\max\{-\bar c,0\}$. Then
\begin{equation}
\label{EQ1.3}
\bar c\ge\delta(k)\quad\text{on } [|\bar p|<k].
\end{equation}
Thus the set on which negative curvature is possible is a subset of the set $[|\bar p|\ge k]$ for a certain $k>0$. The problem of proving \eqref{EQ1.2} is therefore reduced to controlling $\|z_{\bar u,h}\|_{L^2([|\bar p|\ge k])}$ by the norm of linearized state over the set where \eqref{EQ1.3} provides positive curvature. That is to obtain estimates for the following kind. For $\tau<k$ and $h\in C_{\bar u}^{\tau}$, we obtain in Theorem \ref{TH4.1} that
\begin{equation}
\label{EQ1.4}
\|z_{\bar u,h}\|_{L^2(\{|\bar p|\ge k\})}^2\le\frac{\widehat C_0}{(k-\tau)^2}\,\|z_{\bar u,h}\|_{L^2(\{|\bar p|<k\})}^2,
\end{equation}
where $\widehat C_0:=\frac{C_\Omega^2\Lambda_AL_{\bar p}^2}{\lambda_A}$. Then we can formulate the sufficient optimality condition
\begin{equation}
\label{EQ1.5}
\delta(k)(k-\tau)^2>\widehat C_0 M(k).
\end{equation}
Under \eqref{EQ1.5}, we find
\[
J''(\bar u)[h,h]\ge\frac{\delta(k)(k-\tau)^2-\widehat C_0 M(k)}{(k-\tau)^2+\widehat C_0}\,\|z_{\bar u,h}\|_{L^2(\Omega)}^2\quad\forall h\in C_{\bar u}^{\tau}.
\]
This condition can be substantially improved. For this, we consider an elliptic operator
\[
Lz:=-\operatorname{div}(A\nabla z)+a(x)z,
\]
where $A$ is symmetric and uniformly elliptic and
$a\in L^\infty(\Omega)$ satisfies $a\ge0$. Let $\phi\ge0$ be a function satisfying $\phi=0$ on $\partial\Omega$, with sufficient regularity in a positive level set $0<\tau_1<\sigma_1$, and assume
\[
|\nabla\phi|\ge\mu>0\quad\text{on }[\tau_1\le\phi\le\sigma_1].
\]
For $\Omega_s:=[\phi>s]$, $\Gamma_s:=[\phi=s]$, and $s\in [\tau_1, \sigma_1]$, we prove that the domains $\Omega_s$ possess a uniform behaviour in the sense that their boundaries admit graph representations that have, with respect to $s$, uniform bounds on radius, gradient, the number of charts, and the number of charts intersections. Using this, we show that there exists $\widehat C>0$, independent of
$s\in[\tau_1,\sigma_1]$, such that every solution of $Lz=0$ in $\Omega_s$ satisfies
\begin{equation}
\label{EQ1.6}
\|z\|_{L^2(\Omega_s)}\le\widehat C\|\gamma_s z\|_{L^2(\Gamma_s)}.
\end{equation}
Together with the coarea formula, \eqref{EQ1.6} gives the following estimate in Theorem \ref{TH5.1}. If $\tau_1\le\tau\le s_1<s_2\le\sigma_1$ and $Lz=0 \quad\text{in }\Omega_\tau$ then
\begin{equation}
\label{EQ1.7}
\|z\|_{L^2(\Omega_{s_2})}^2\le\exp\left(-\frac{s_2-s_1}{C_0}\right)\|z\|_{L^2(\Omega_{s_1})}^2,
\end{equation}
where $C_0:=\widehat C^2 L_*$, $L_*:=\max\{\|\nabla\phi\|_{L^\infty(\Omega)},\mu\}$. 
We apply this result with $\phi=|\bar p|$, $a=f_y(\cdot,\bar y)$. If $h\in C_{\bar u}^{\tau}$, then $h=0$ a.e. on $[|\bar p|>\tau]$. The corresponding linearized state satisfies
\[
-\operatorname{div}(A\nabla z_{\bar u,h})+f_y(\cdot,\bar y)z_{\bar u,h}=0\quad\text{in }[|\bar p|>\tau].
\]
Therefore, the above applies to the linearized states appearing in the second-order condition. Let $\tau_1\le\tau<k\le\sigma_1$ and define $q(\tau,k):=\exp\left(-\frac{k-\tau}{C_0}\right)$. Then
\begin{equation}
\label{EQ1.8}
\|z_{\bar u,h}\|_{L^2([|\bar p|\ge k])}^2\le q(\tau,k)\|z_{\bar u,h}\|_{L^2(\Omega)}^2.
\end{equation}
Combining \eqref{EQ1.8} with \eqref{EQ1.3} and the definition of $M(k)$ gives
\[
J''(\bar u)[h,h]\ge\kappa(\tau,k)\|z_{\bar u,h}\|_{L^2(\Omega)}^2,
\]
where $\kappa(\tau,k):=\delta(k)-\bigl(\delta(k)+M(k)\bigr)\exp\left(-\frac{k-\tau}{C_0}\right)$. Therefore $\kappa(\tau,k)>0$ if and only if
\begin{equation}
\label{EQ1.9}
M(k)<\delta(k)\left[\exp\left(\frac{k-\tau}{C_0}\right)-1\right].
\end{equation}
When \eqref{EQ1.9} holds, \eqref{EQ1.2} follows on the extended critical cone and the second-order theory implies local quadratic growth with respect to the linearized states: there exist $\rho,\gamma>0$ such that
\[
J(u)\ge J(\bar u) +\frac{\gamma}{2}\|y_u-\bar y\|_{L^2(\Omega)}^2
\]
for controls $u$ for which the state $y_u$ is sufficiently close to $y_{\bar u}$ in $L^\infty(\Omega)$.
Utilizing the insights obtained in the above, we show that the sufficient second-order optimality condition \eqref{EQ1.9} is stable under perturbations of the optimal control problem.

We summarize the contributions of the paper as follows.

\begin{enumerate}
\item We present a necessary property for a stationary point to satisfy the well known second-order condition. This condition implies that the linearized states associated with the extended critical directions cannot concentrate arbitrarily on regions of small or negative curvature, and is the motivation of the following analysis.
\item Without additional geometric assumption on the adjoint state, we obtain the quantitative estimate \eqref{EQ1.4} and the sufficient condition \eqref{EQ1.5}.
\item Under a gradient bound of the adjoint state on certain level sets of the adjoint state, solutions of a homogeneous elliptic equation satisfy the estimates \eqref{EQ1.6} and \eqref{EQ1.7}. The constants are uniform with respect to the level set parameter. Applied with $\phi=|\bar p|$, the theorem gives \eqref{EQ1.9}. This allows the pointwise curvature coefficient to be negative on a subset of positive measure for achieving the state-coercive second-order sufficient condition.
\item We show that the obtained coercivity of the second variation is stable under perturbations. If the stationary points, states, adjoints, and tracking data of perturbed problems converge to a reference point satisfying the above conditions, then coercivity of the second variation persists for all sufficiently small perturbations.
\end{enumerate}

The paper is organized as follows. The Section \ref{sec2} contains standing assumptions and introduces the optimal control
problem, its first- and second-order derivatives, and the relevant critical cones. Section \ref{sec3} establishes the necessary condition for the coercivity of the second variation. In Section \ref{sec4} a first quantitative sufficient second-order optimality condition is obtained. In the following Section \ref{sec5} the elliptic theory, including uniform estimates on quantities relating to the geometry of a family of level sets is derived. Then this abstract result is applied to the optimal control problem in Section \ref{sec6}. Following this, in Section \ref{sec7} stability of this optimality condition under perturbations is proved. Technical results concerning Lipschitz domains, traces, extension operators, conormal derivatives, and the coarea formula are collected in the appendix.

\section{Standing assumptions and preliminaries}
\label{sec2}

In this section, we collect the standing assumptions and some of the properties of the state equation, the adjoint equation, and the first-and second-order Fr\'{e}chet derivatives of the objective functional. We also introduce the critical cones and the notation used.
\begin{assumption}
\label{A2.1}
The following conditions hold.
\begin{enumerate}
\item $\Omega\subset\mathbb R^n$, $n\in\{2,3\}$, is a bounded domain with $C^{1,1}$ boundary $\Gamma$.
\item The matrix-valued map $A=(a_{ij})\in C^{0,1} \bigl(\overline{\Omega};\mathbb R^{n\times n}\bigr)$ is symmetric, and there are constants $\lambda_A,\Lambda_A>0$ such that for every $\xi\in\mathbb R^n$ and for a.e. $x\in\Omega$
\[
\lambda_A|\xi|^2 \le A(x)\xi\cdot\xi \le \Lambda_A|\xi|^2.
\]
\item
Following \cite{CasasMateos2021}, we assume that $f:\Omega\times\mathbb R\to\mathbb R$ is a Carath\'eodory function and that $f(x,\cdot)$ is of class $C^2$ for almost every $x\in\Omega$. Moreover, for a.e. $x\in\Omega$ and every $y\in\mathbb R$,
\[
f(\cdot,0)\in L^\infty(\Omega), \quad f_y(x,y)\ge0.
\]

For every $M>0$, there exists a constant $C_{f,M}>0$ such that
\[
|f_y(x,y)|+|f_{yy}(x,y)|\le C_{f,M}
\]
for a.e. $x\in\Omega$ and every $|y|\le M$.

For every $M>0$ and every $\varepsilon>0$, there exists
$\delta>0$ such that
\[
|f_{yy}(x,y_1)-f_{yy}(x,y_2)|<\varepsilon
\]
for almost every $x\in\Omega$ and all
$y_1,y_2\in[-M,M]$ satisfying $|y_1-y_2|<\delta$. 
\item
For one fixed exponent $r>n$, which may be chosen arbitrarily close
to $n$, the desired state satisfies $y_d\in L^r(\Omega)$. 
\end{enumerate}
\end{assumption}
The following result concerning the state equation is taken from \cite[Theorem 1]{CasasMateos2021}.

\begin{theorem}[\cite{CasasMateos2021}]
\label{TH2.1}
Let Assumption \ref{A2.1} be satisfied. For every $u\in L^p(\Omega)$ with $p>\frac n2$ there exists a unique solution $y_u\in Y:= H_0^1(\Omega)\cap C(\overline{\Omega})$ of \eqref{EQ1.1}. Moreover, there exists a constant $T_p>0$, independent of $u$, such that
\[
\|y_u\|_{H_0^1(\Omega)}+ \|y_u\|_{C(\overline{\Omega})}\le T_p \left(\|u\|_{L^p(\Omega)}+ \|f(\cdot,0)\|_{L^\infty(\Omega)}\right).
\]
If $u_k\rightharpoonup u \quad\text{weakly in }L^p(\Omega)$, then
\[
\|y_{u_k}-y_u\|_{C(\overline{\Omega})} + \|y_{u_k}-y_u\|_{H_0^1(\Omega)} \to0.
\]
If, in addition, $u\in L^\infty(\Omega)$, then $y_u\in W^{2,p}(\Omega)$ for every finite $p$ and
\begin{equation}
\label{EQ2.1}
\|y_u\|_{W^{2,p}(\Omega)}\le M_0p\left(\|u\|_{L^\infty(\Omega)} +\|f(\cdot,0)\|_{L^\infty(\Omega)}\right),
\end{equation}
where $M_0$ is independent of $u$ and $p$.
\end{theorem}

Since the admissible controls are uniformly bounded in
$L^\infty(\Omega)$, \eqref{EQ2.1} implies the existence of a
constant $C>0$ such that
\[
\|y_u\|_{C(\overline{\Omega})}\le C\quad\forall u\in\mathcal U.
\]

For $p>n/2$, we denote by
\[
G:L^p(\Omega)\to Y
\]
the control-to-state operator defined by $G(u):=y_u$, where $y_u$ denotes the solution of \eqref{EQ1.1}. The differentiability properties of $G$ are stated next. The proof is
obtained by the implicit-function theorem, see, for example,
\cite[Theorem 1]{zbMATH06853843}.

\begin{theorem}
\label{TH2.2}
The control-to-state operator $G$ is of class $C^2$. For
$u,v,w\in L^p(\Omega)$, $p>n/2$, the linearized state $z_v:=G'(u)v$ is the unique solution of
\begin{equation*}
-\operatorname{div}(A\nabla z_v)+f_y(\cdot,y_u)z_v=v\quad\text{in }\Omega, \quad z_v=0 \quad\text{on }\Gamma.
\end{equation*}
Moreover, $z_{v,w}:=G''(u)(v,w)$ is the unique solution of
\[
-\operatorname{div}(A\nabla z_{v,w})+f_y(\cdot,y_u)z_{v,w}+f_{yy}(\cdot,y_u)z_vz_w=0
\quad\text{in }\Omega, \quad z_{v,w}=0 \quad\text{on }\Gamma.
\]
\end{theorem}
For every control $u$, we define the corresponding adjoint state
$p_u$ as the solution of
\begin{equation}
\label{EQ2.2}
-\operatorname{div}(A\nabla p_u) + f_y(\cdot,y_u)p_u=y_u-y_d\quad\text{in }\Omega,\quad p_u=0 \quad\text{on }\Gamma.
\end{equation}
Assumption \ref{A2.1}, together with Theorem \ref{TH2.1}, implies $p_u \in H_0^1(\Omega)\cap C(\overline{\Omega})$ for all $u\in\mathcal U$. Moreover, for every $2<q\le r$, the adjoint satisfies $p_u\in W^{2,q}(\Omega)$ with $\|p_u\|_{W^{2,q}(\Omega)} \le C_q$ for all $u\in\mathcal U$, where $C_q$ is independent of $u$. There exists $T_\infty>0$ such that
\[
\|p_u\|_{H_0^1(\Omega)}+\|p_u\|_{C(\overline{\Omega})}\le T_\infty \quad \forall u\in\mathcal U.
\]
In what follows, $\bar u\in\mathcal U$ denotes a fixed stationary
control, and we write $\bar y:=y_{\bar u}$, $\bar p:=p_{\bar u}$. For the fixed $r>n$ from Assumption \ref{A2.1},  we set $\vartheta:=1-\frac nr \in(0,1)$. By the Sobolev-Morrey embedding \cite[Theorem 7.26(ii)]{GilbargTrudinger2001},
\begin{equation}
\label{EQ2.3}
\bar p \in C^{1,\vartheta}(\overline{\Omega}), \quad \|\bar p\|_{C^{1,\vartheta}(\overline{\Omega})}
\le C \|\bar p\|_{W^{2,r}(\Omega)}.
\end{equation}
If $y_d\in L^\infty(\Omega)$, the above regularity can be applied for every finite
$r>n$, and therefore $\bar p \in C^{1,\theta}(\overline{\Omega})$ for all $\theta\in(0,1)$. We define
\begin{equation}
\label{EQ2.4} 
L_{\bar p}:=\|\nabla\bar p\|_{L^\infty(\Omega)}, \quad H_{\bar p}:=[\nabla\bar p]_{C^{0,\vartheta}(\overline{\Omega})}.
\end{equation}
Both constants are finite by \eqref{EQ2.3}. The following differentiability result is implied by the chain rule, Theorem \ref{TH2.2}, and Assumption \ref{A2.1}.
\begin{theorem}
\label{TH2.3}
For $p>n/2$, the objective functional $J:L^p(\Omega)\to\mathbb R$ is of class $C^2$. For every
$u,v,w\in L^p(\Omega)$,
\begin{align}
J'(u)v&=\int_\Omega p_u v\,dx\label{EQ2.5}\\
J''(u)[v,w]&=\int_\Omega\left(1-p_u f_{yy}(\cdot,y_u)\right)z_vz_w\,dx.\label{EQ2.6}
\end{align}
Here $p_u$ is the solution of \eqref{EQ2.2} and $z_v=G'(u)v$, $z_w=G'(u)w$.
\end{theorem}
We now fix the exponent $r>n$ from Assumption \ref{A2.1} and the corresponding exponent
$\vartheta=1-n/r$. A control $\bar u\in\mathcal U$ is called stationary if
\[
J'(\bar u)(u-\bar u)=\int_\Omega \bar p(u-\bar u)\,dx \ge0 \quad \forall u\in\mathcal U.
\]
This is equivalent to the pointwise conditions
\begin{equation}
\label{EQ2.7}
\bar p\ge0 \quad\text{a.e. on }[\bar u=\alpha], \quad \bar p=0 \quad\text{a.e. on }[\alpha<\bar u<\beta], \quad \bar p\le0 \quad\text{a.e. on }[\bar u=\beta].
\end{equation}
The tangent cone to $\mathcal U$ at $\bar u$ is
\[
T_{\mathcal U}(\bar u)= \left\{h\in L^2(\Omega):\begin{array}{l}h\ge0\quad\text{a.e. on }[\bar u=\alpha], \\ 
h\le0 \quad\text{a.e. on }[\bar u=\beta]\end{array}\right\}.
\]
For every $h\in T_{\mathcal U}(\bar u)$, the term $\bar p\,h$ is nonnegative almost everywhere by \eqref{EQ2.7}. The critical cone is
\[
C_{\bar u}:= \left\{h\in T_{\mathcal U}(\bar u):h=0\quad\text{a.e. on }[\bar p\ne0]\right\}.
\]
For $\tau>0$, we define 
\begin{equation}
\label{EQ2.8}
K_{\bar p}^{\tau}:= \left\{h\in L^2(\Omega):h=0\quad\text{a.e. on }[|\bar p|>\tau]\right\}
\end{equation}
and the extended critical cone
\[
C_{\bar u}^{\tau}:=T_{\mathcal U}(\bar u)\cap K_{\bar p}^{\tau}.
\]
To ease the notation, we make the following definitions, which will be
used throughout the remainder of the paper.

\begin{definition}
\label{DE2.1}
Let $m:=\|f_{yy}(\cdot,\bar y)\|_{L^\infty(\Omega)}$ and let $c^-:=\max\{-c,0\}$. If $m>0$, then for every $0<k<\frac1m$ we set $\delta(k):=1-mk$ and $M(k):=\|c^-\|_{L^\infty([|\bar p|\ge k])}$. For $\varepsilon\in(0,1)$, we define $k_\varepsilon:=\frac{1-\varepsilon}{m}$ and $M_\varepsilon:=\|c^-\|_{L^\infty([k_\varepsilon\le|\bar p|])}$. Then $\delta(k_\varepsilon)=\varepsilon$, $M(k_\varepsilon) =M_\varepsilon$. For $s>0$, we set
\begin{equation*}
\Omega_s:=[|\bar p|>s],\quad\Gamma_s:=[|\bar p|=s].
\end{equation*}
Finally, for the fixed reference state $\bar y$, we write $\mathcal L:=-\operatorname{div}(A\nabla(\cdot))+f_y(\cdot,\bar y)(\cdot)$. 
\end{definition}
The value $m$ is finite because $\bar y\in C(\overline{\Omega})$ and $f_{yy}$ is uniformly bounded on bounded subsets of its second argument by Assumption \ref{A2.1}. If $m=0$ then $f_{yy}(\cdot,\bar y)=0$ a.e. in $\Omega$, and $c\equiv1$. Therefore,
\[
J''(\bar u)[h,h]=\|z_{\bar u,h}\|_{L^2(\Omega)}^2\quad \forall h\in L^2(\Omega).
\]
Thus the only nontrivial case for the following analysis is $m>0$, which we assume from now on.
\begin{lemma}
\label{LE2.1}
Let $\varepsilon\in(0,1)$. Then $c\ge\varepsilon$ on $[|\bar p|<k_\varepsilon]$,  and
\[
[c<\varepsilon]\subset[k_\varepsilon<|\bar p|] \subset [k_\varepsilon\le|\bar p|]\Subset\Omega,
\]
and $c\ge-M_\varepsilon$ on $[k_\varepsilon\le|\bar p|]$.
\end{lemma}
\begin{proof}
If $x\in[|\bar p|<k_\varepsilon]$, then
\[
c(x)=1-\bar p(x)f_{yy}(x,\bar y(x))\ge1-|\bar p(x)| |f_{yy}(x,\bar y(x))|>1-k_\varepsilon m=\varepsilon.
\]
Thus $c\ge\varepsilon$ on $[|\bar p|<k_\varepsilon]$. If $c(x)<\varepsilon$, then
\[
1-\varepsilon<\bar p(x)f_{yy}(x,\bar y(x))\le|\bar p(x)|m.
\]
Thus $|\bar p(x)|>\frac{1-\varepsilon}{m}=k_\varepsilon$, and therefore $[c<\varepsilon]\subset [k_\varepsilon<|\bar p|]$. The set $[k_\varepsilon\le|\bar p|]$ is compact in $\overline{\Omega}$. Since $\bar p=0$ on $\Gamma$ and $k_\varepsilon>0$, it is disjoint from $\Gamma$. Therefore, $[k_\varepsilon\le|\bar p|] \Subset \Omega$. Further, from the definition of $M_\varepsilon$, $c\ge-M_\varepsilon$ on $[k_\varepsilon\le|\bar p|]$. 
\end{proof}
This can be written as
\begin{equation}
\label{EQ2.9}
c \ge \delta(k)=1-mk\quad\text{on }[|\bar p|<k],
\end{equation}
and $c\ge-M(k)$ on $[|\bar p|\ge k]$. Later, we take numbers $0<\tau_1<\sigma_1$ and choose the parameters so that
\[
0<\varepsilon<1, \quad\tau_1\le\tau<k_\varepsilon\le\sigma_1.
\]
or in the same way, $\tau_1 \le \tau <k \le\sigma_1$. 
\begin{remark}
The $C^{1,1}$ regularity of $\partial\Omega$ is imposed to simplify the regularity statements, the presentation and tracking of the
needed constants. The main argument can also be obtained for bounded Lipschitz domains, with modified constants. The continuity of $\bar p$ on $\overline\Omega$ and its homogeneous Dirichlet boundary condition imply that, for every $\tau_1>0$, $[|\bar p|\ge\tau_1/2]\Subset\Omega$. Interior elliptic regularity gives the $C^{1,\vartheta}$ bounds that are needed for the level-set argument.
\end{remark}

\section{A necessary condition for coercivity}
\label{sec3}

The sufficient conditions developed in the following sections use the fact, that the linearized states associated with critical directions cannot concentrate arbitrarily in regions where the curvature is small or negative. We first observe that this property is also necessary whenever the classical second-order coercivity holds. Let $\bar u$ be stationary and let $\bar y$ and $\bar p$ be the corresponding state and adjoint state. We recall that $c=1-\bar p f_{yy}(\cdot,\bar y)$. 
\begin{theorem}
Assume that there exist $\tau>0$ and $\kappa>0$ such that
\begin{equation}
\label{EQ3.1}
J''(\bar u)[v,v]=\int_\Omega c\,z_v^2\,dx\ge\kappa\|z_v\|_{L^2(\Omega)}^2\quad\forall v\in C_{\bar u}^{\tau}.
\end{equation}
Then every $v\in C_{\bar u}^{\tau}$ satisfies
\begin{equation}
\label{EQ3.2}
\int_{[c<\kappa]}(\kappa-c)z_v^2\,dx\le\int_{[c>\kappa]}(c-\kappa)z_v^2\,dx.
\end{equation}
For $\gamma<\kappa$, define $E_\gamma:=\{x\in\Omega:c(x)\le\gamma\}$. Then
\begin{equation}
\label{EQ3.3}
\|z_v\|_{L^2(E_\gamma)}^2
\le\frac{\|(c-\kappa)^+\|_{L^\infty(\Omega)}}{\kappa-\gamma}\|z_v\|_{L^2(\Omega\setminus E_\gamma)}^2\quad\forall v\in C_{\bar u}^{\tau}.
\end{equation}

\end{theorem}
\begin{proof}
Let $v\in C_{\bar u}^{\tau}$ and $z:=z_v$. By \eqref{EQ3.1},
\[
0\le J''(\bar u)[v,v]-\kappa\|z\|_{L^2(\Omega)}^2=\int_\Omega(c-\kappa)z^2\,dx.
\]
Splitting the integral into the sets where $c$ is above and below
$\kappa$ gives
\[
0\le\int_{[c>\kappa]}(c-\kappa)z^2\,dx-\int_{[c<\kappa]}(\kappa-c)z^2\,dx.
\]
This proves
\eqref{EQ3.2}.
Let $\gamma<\kappa$. On $E_\gamma$ we have $\kappa-c\ge\kappa-\gamma$. Moreover, $E_\gamma\subset[c<\kappa]$. Therefore,
\begin{align*}
(\kappa-\gamma)\|z\|_{L^2(E_\gamma)}^2&\le\int_{E_\gamma}(\kappa-c)z^2\,dx\le\int_{[c<\kappa]}(\kappa-c)z^2\,dx\\
&\le\int_{[c>\kappa]}(c-\kappa)z^2\,dx\le\|(c-\kappa)^+\|_{L^\infty(\Omega)}\,\|z\|_{L^2([c>\kappa])}^2.
\end{align*}
Since $[c>\kappa]\subset\Omega\setminus E_\gamma$, we obtain
\[
(\kappa-\gamma)\|z\|_{L^2(E_\gamma)}^2\le\|(c-\kappa)^+\|_{L^\infty(\Omega)}\,\|z\|_{L^2(\Omega\setminus E_\gamma)}^2.
\]
\end{proof}

\section{A first sufficient condition for coercivity}
\label{sec4}

As a first result, we obtain the following.

\begin{theorem}
\label{TH4.1}
Let Assumption \ref{A2.1} be satisfied and let
$\bar u\in\mathcal U$ be stationary. Let $\widehat C_0:=\frac{C_{\Omega}^2\Lambda_A L_{\bar p}^2}{\lambda_A}$. Then, for every $0<\tau<k<\frac1m$ and every $h\in C_{\bar u}^{\tau}$,
\begin{equation}
\label{E4.1}
\|z_{\bar u,h}\|_{L^2([k\leq|\bar p|])}^2\leq\frac{\widehat C_0}{(k-\tau)^2}\|z_{\bar u,h}\|_{L^2([|\bar p|<k])}^2.
\end{equation}
If $\delta(k)(k-\tau)^2>\widehat C_0 M(k)$, then
\[
J''(\bar u)[h,h]\geq\kappa_{\tau,k}\|z_{\bar u,h}\|_{L^2(\Omega)}^2\quad\forall h\in C_{\bar u}^{\tau},
\]
where
\[
\kappa_{\tau,k}:=\frac{\delta(k)(k-\tau)^2-\widehat C_0M(k)}{(k-\tau)^2+\widehat C_0}>0.
\]
\end{theorem}

\begin{proof}
Given $h\in C_{\bar u}^{\tau}$ and $z:=z_{\bar u,h}$, consider $\zeta:=\min\left\{1,\frac{(|\bar p|-\tau)_+}{k-\tau}\right\}$. 
Since $|\bar p|\in W^{1,\infty}(\Omega)$, we have $\zeta\in W^{1,\infty}(\Omega)$, $0\leq\zeta\leq1$,
\[
\zeta=0\quad\text{on }[|\bar p|\leq\tau], \quad \zeta=1\quad\text{on }[k\leq|\bar p|],
\]
and
\begin{equation}
\label{E4.2}
\nabla\zeta=\frac{1}{k-\tau}\nabla|\bar p|\,\chi_{[\tau<|\bar p|<k]}\quad\text{a.e. in }\Omega.
\end{equation}
We estimate $|\nabla\zeta|\leq\frac{L_{\bar p}}{k-\tau}\chi_{[\tau<|\bar p|<k]}$. Because $z\in H_0^1(\Omega)$ and $\zeta\in W^{1,\infty}(\Omega)$, the functions $\zeta z$ and $\zeta^2z$ belong to $H_0^1(\Omega)$. Since $h=0$ almost everywhere on $[|\bar p|>\tau]$ and $\zeta=0\quad\text{on }[|\bar p|\leq\tau]$, we infer that
\[
\int_\Omega h\zeta^2z\,dx=0.
\]
For $w,\varphi\in H_0^1(\Omega)$, let
\[
a_\Omega(w,\varphi):=\int_\Omega\left(A\nabla w\cdot\nabla\varphi+f_y(\cdot,\bar y)w\varphi\right)\,dx.
\]
Testing the equation solved by $z$ with $\zeta^2z$ gives $a_\Omega(z,\zeta^2z)=0$. Furthermore,
\begin{align}
a_\Omega(\zeta z,\zeta z)&=a_\Omega(z,\zeta^2z)+\int_\Omega z^2 A\nabla\zeta\cdot\nabla\zeta\,dx\nonumber=\int_\Omega z^2 A\nabla\zeta\cdot\nabla\zeta\,dx.
\end{align}
From the ellipticity of $A$ and the fact that
$f_y(\cdot,\bar y)\ge0$, we have
\[
a_\Omega(\zeta z,\zeta z)\geq\lambda_A\|\nabla(\zeta z)\|_{L^2(\Omega)}^2.
\]
Using the Poincar\'e inequality, $\|\zeta z\|_{L^2(\Omega)}\leq C_\Omega \|\nabla(\zeta z)\|_{L^2(\Omega)}$, we obtain
\[
a_\Omega(\zeta z,\zeta z)\geq\frac{\lambda_A}{C_\Omega^2}\|\zeta z\|_{L^2(\Omega)}^2.
\]
On the other hand, \eqref{E4.2} gives
\begin{align*}
a_\Omega(\zeta z,\zeta z)
&\leq\Lambda_A\int_\Omega z^2|\nabla\zeta|^2\,dx\leq\frac{\Lambda_A L_{\bar p}^2}{(k-\tau)^2}\int_{[\tau<|\bar p|<k]}z^2\,dx\\
&\leq\frac{\Lambda_A L_{\bar p}^2}{(k-\tau)^2}\|z\|_{L^2([|\bar p|<k])}^2.
\end{align*}

Because $\zeta=1$ on $[k\leq|\bar p|]$, it follows that
\begin{align*}
\|z\|_{L^2([k\leq|\bar p|])}^2&\leq\|\zeta z\|_{L^2(\Omega)}^2\leq\frac{C_\Omega^2}{\lambda_A}a_\Omega(\zeta z,\zeta z)\leq\frac{\widehat C_0}
{(k-\tau)^2}\|z\|_{L^2([|\bar p|<k])}^2.
\end{align*}
This proves \eqref{E4.1}. From \eqref{E4.1}, it follows that
\begin{align*}
\|z\|_{L^2([k\leq|\bar p|])}^2&\leq\frac{\widehat C_0}{(k-\tau)^2}\Bigl(\|z\|_{L^2([k\leq|\bar p|])}^2+\|z\|_{L^2([|\bar p|<k])}^2-\|z\|_{L^2([k\leq|\bar p|])}^2\Bigr),
\end{align*}
which implies
\begin{align}
\|z\|_{L^2([k\leq|\bar p|])}^2\leq\frac{\widehat C_0}
{(k-\tau)^2+\widehat C_0}\Bigl(\|z\|_{L^2([k\leq|\bar p|])}^2+\|z\|_{L^2([|\bar p|<k])}^2\Bigr).
\label{EQ4.3}
\end{align}
By \eqref{EQ2.9}, we have $c\geq\delta(k)$ on $[|\bar p|<k]$ and, by the definition of $M(k)$, $c\geq-M(k)$ on $[k\leq|\bar p|]$. Therefore,
\begin{align*}
J''(\bar u)[h,h]&=\int_\Omega cz^2\,dx\geq\delta(k)\|z\|_{L^2([|\bar p|<k])}^2-M(k)\|z\|_{L^2([k\leq|\bar p|])}^2\\
&=\delta(k)\Bigl(\|z\|_{L^2([k\leq|\bar p|])}^2+\|z\|_{L^2([|\bar p|<k])}^2\Bigr)-\bigl(\delta(k)+M(k)\bigr)\|z\|_{L^2([k\leq|\bar p|])}^2.
\end{align*}
Combining the above estimate with \eqref{EQ4.3}, we obtain
\begin{align*}
J''(\bar u)[h,h]&\geq\left[\delta(k)-\bigl(\delta(k)+M(k)\bigr)\frac{\widehat C_0}{(k-\tau)^2+\widehat C_0}\right]\\
&\quad\cdot\Bigl(\|z\|_{L^2([k\leq|\bar p|])}^2+\|z\|_{L^2([|\bar p|<k])}^2\Bigr)\\
&=\frac{\delta(k)(k-\tau)^2-\widehat C_0M(k)}{(k-\tau)^2+\widehat C_0}\|z\|_{L^2(\Omega)}^2.
\end{align*}
\end{proof}

In the next sections, we show that a stronger result is available if the gradient of the adjoint state is bounded away from zero on a sufficient large set of level sets. 

\section{Elliptic estimates on level sets}
\label{sec5}

We first consider the elliptic regularity theory and general properties of solutions to linear elliptic equations. 
\begin{assumption}
\label{AS5.1}
Let $\Omega\subset\mathbb R^n$, $n\in\{2,3\}$, be a bounded domain with
$C^{1,1}$ boundary and let
\[
Lz:=-\operatorname{div}(A\nabla z)+az,
\]
where $A\in C^{0,1}\bigl(\overline{\Omega};\mathbb R^{n\times n}\bigr)$, $A=A^\top$, and there exist constants $0<\lambda_A\le\Lambda_A$ such that
\[
\lambda_A|\xi|^2\le A(x)\xi\cdot\xi\le\Lambda_A|\xi|^2
\]
for a.e. $x\in\Omega$ and every $\xi\in\mathbb R^n$. The function $a$ satisfies $a\in L^\infty(\Omega)$, $a\ge0$ a.e. in $\Omega$. 

Let $\phi\in W^{1,\infty}(\Omega)\cap C(\overline{\Omega})$ be such that the following holds
\[
\phi\ge0\quad\text{in }\Omega, \quad\phi=0\quad\text{on }\partial\Omega.
\]
For $0<\tau_1<\sigma_1$ we assume that, on a neighborhood of $\{x\in\overline{\Omega}:\phi(x)\ge\tau_1/2\}$, the function $\phi$ belongs to $C^{1,\vartheta}$ for some $\vartheta\in(0,1)$. We write  $L_\phi:=\|\nabla\phi\|_{L^\infty(\Omega)}$ and assume that $[\nabla\phi]_{C^{0,\vartheta}}\le H_\phi$  in this neighborhood. We assume that the gradient of $\phi$ has the property that there exists $\mu>0$ such that
\[
\mu\le|\nabla\phi(x)|\quad\text{for every }x\in\overline{\Omega}\text{ with }\tau_1\le\phi(x)\le\sigma_1.
\]
\end{assumption}

\begin{definition}
For $s>0$, we define
\begin{equation*}
\Omega_s:=\{x\in\Omega:\phi(x)>s\},\quad\Gamma_s:=\{x\in\Omega:\phi(x)=s\}.
\end{equation*}
\end{definition}

\begin{definition}
We introduce the constants
\begin{equation*}
\begin{aligned}
L_*&:=\max\{L_\phi,\mu\},\quad  r_*:=\min\left\{\frac{\tau_1}{2L_*},\left(\frac{\mu}{2H_\phi}\right)^{1/\vartheta}\right\},\\
\rho_*&:=\frac{\mu r_*}{8L_*},\quad M_*:=\frac{2^{2+\vartheta/2}H_\phi}{\mu}.
\end{aligned}
\end{equation*}
If $H_\phi=0$, then $\left(\frac{\mu}{2H_\phi}\right)^{1/\vartheta}$ is understood as $+\infty$. Further, $r_*,\rho_*>0$,  $\rho_*\le\frac{r_*}{8}$.
\end{definition}

\subsection{Geometry of the level sets}
The following lemma is standard up to the uniformity of the constants with respect to the parameter $s$. We provide details in to show the dependence of the involved constants. We refer to \cite[Chapter 3]{zbMATH05816515} and \cite[Chapter 1]{Necas2012} for related discussions.
\begin{lemma}
\label{LE5.1}
For every $s\in[\tau_1,\sigma_1]$, the following statements hold.
\begin{enumerate}
\item[(a)]
Either $\Omega_s=\Gamma_s=\varnothing$, or $\Gamma_s$ is a compact $C^{1,\vartheta}$-hypersurface with $\partial\Omega_s=\Gamma_s$.
\item[(b)]
Let $x_0\in\Gamma_s$. After translating $x_0$ to the origin and
rotating the coordinates so that $e_n=\frac{\nabla\phi(x_0)}{|\nabla\phi(x_0)|}$, there exists $\psi\in C^{1,\vartheta}(B'_{\rho_*})$ and
\[
\mathcal C_*:=B'_{\rho_*}\times\left(-\frac{r_*}{2},\frac{r_*}{2}\right),
\]
with 
\begin{equation}
\label{EQ5.1}
\begin{aligned}
\Gamma_s\cap\mathcal C_*&=\{(x',x_n)\in\mathcal C_*:x_n=\psi(x')\},\\ 
\Omega_s\cap\mathcal C_*&=\{(x',x_n)\in\mathcal C_*:x_n>\psi(x')\},
\end{aligned}
\end{equation}
and
\begin{equation}
\label{EQ5.2}
\psi(0)=0,\quad\nabla'\psi(0)=0,\quad\|\nabla'\psi\|_{L^\infty(B'_{\rho_*})}\le1,\quad[\nabla'\psi]_{C^{0,\vartheta}(B'_{\rho_*})}\le M_*.
\end{equation}
\item[(c)]
The outward unit normal to $\Omega_s$ is $\nu_s=-\frac{\nabla\phi}{|\nabla\phi|}$. 
\item[(d)]
There exists an integer $N_*$, independent of $s$, such that $\Gamma_s$ can be covered by at most $N_*$ charts from (b). 
\end{enumerate}
\end{lemma}

\begin{proof}
We first prove statements (a), (b), and (c). Let $s\in[\tau_1,\sigma_1]$ be given and assume that $\Gamma_s=\varnothing$ but $\Omega_s\ne\varnothing$. Then $\phi(x)>s$ for some $x\in\Omega_s$. We take $y\in\partial\Omega$ with $|x-y|=\operatorname{dist}(x,\partial\Omega)$, which we can do since $\overline{\Omega}$ is compact. Then we define $x_t:=(1-t)x+ty$, $t\in[0,1]$. For $t\in[0,1)$, $|x-x_t|=t|x-y|<\operatorname{dist}(x,\partial\Omega)$, $x_t\in\Omega$. Since $g(t):=\phi(x_t)$ is continuous,
\[
g(0)=\phi(x)>s,\quad g(1)=\phi(y)=0<s.
\]
There exists $\bar t\in(0,1)$ such that $g(\bar t)=s$, which implies $x_{\bar t}\in\Gamma_s$. This contradicts $\Gamma_s=\varnothing$. Thus $\Gamma_s=\varnothing$ $\to$ $\Omega_s=\varnothing$. Let now $\Gamma_s\ne\varnothing$ and $x_0\in\Gamma_s$. Then $\phi(x_0)=s$. We choose $y\in\partial\Omega$ such that $|x_0-y|=\operatorname{dist}(x_0,\partial\Omega)$. Since $\phi(y)=0$, $s=|\phi(x_0)-\phi(y)|$. Using the Lipschitz continuity of $\phi$, we obtain $s\le L_\phi|x_0-y|\le L_*\operatorname{dist}(x_0,\partial\Omega)$. Then taking the infimum over $x_0\in\Gamma_s$, 
\[
\operatorname{dist}(\Gamma_s,\partial\Omega) \ge \frac{s}{L_*} \ge \frac{\tau_1}{L_*}.
\]
Since $r_*\le\frac{\tau_1}{2L_*}$, we have $2r_*\le\operatorname{dist}(\Gamma_s,\partial\Omega)$, and $B_{r_*}(x_0)\subset\Omega$. 
For $x\in B_{r_*}(x_0)$,
\[
|\phi(x)-\phi(x_0)|\le L_\phi|x-x_0| \le L_\phi r_* \le \frac{\tau_1}{2}.
\]
Since $\phi(x_0)=s\ge\tau_1$, we have that $\phi(x)\ge\frac{\tau_1}{2}$ for all $x\in B_{r_*}(x_0)$. Then the $C^{1,\vartheta}$-regularity from Assumption \ref{AS5.1} holds on $B_{r_*}(x_0)$, and $|\nabla\phi(x)-\nabla\phi(y)| \le H_\phi|x-y|^\vartheta$ for all $\forall x,y\in B_{r_*}(x_0)$. We translate $x_0$ to the origin and rotate the coordinates such that
\[
e_n=\frac{\nabla\phi(x_0)}{|\nabla\phi(x_0)|}.
\]
In these coordinates, $\phi(0)=s$, $\nabla'\phi(0)=0$, $\partial_n\phi(0)=|\nabla\phi(0)|$.
Since $\phi(0)=s\in[\tau_1,\sigma_1]$, Assumption \ref{AS5.1} implies $|\nabla\phi(0)|\ge\mu$, and $\partial_n\phi(0)\ge\mu$. For $x\in B_{r_*}(0)$, the H\"older estimate and the definition of
$r_*$ give
\[
|\nabla\phi(x)-\nabla\phi(0)|
\le
H_\phi|x|^\vartheta\le H_\phi r_*^\vartheta \le \frac{\mu}{2}.
\]
For the normal derivative, we therefore obtain
\[
\partial_n\phi(x)\ge\partial_n\phi(0)-|\partial_n\phi(x)-\partial_n\phi(0)|\ge\mu-
|\nabla\phi(x)-\nabla\phi(0)|\ge\frac{\mu}{2}.
\]
For the tangential part, since $\nabla'\phi(0)=0$,
\[
|\nabla'\phi(x)|=|\nabla'\phi(x)-\nabla'\phi(0)|
\le |\nabla\phi(x)-\nabla\phi(0)|\le \frac{\mu}{2}.
\]
We recall that $\mathcal C_*=B'_{\rho_*}\times\left(-\frac{r_*}{2},\frac{r_*}{2}\right)$ and take $x=(x',x_n)\in\mathcal C_*$. Then $|x'|\le\rho_*$, $|x_n|\le\frac{r_*}{2}$, and $|x|^2\le\rho_*^2+\frac{r_*^2}{4}$. Since $\rho_*\le\frac{r_*}{8}$, we have $|x|^2\le\frac{r_*^2}{64}+\frac{r_*^2}{4}<r_*^2$, thus $\mathcal C_*\subset B_{r_*}(0)$. Given $x'\in B'_{\rho_*}$, since $\phi(0)=s$,
\[
|\phi(x',0)-s|
=
|\phi(x',0)-\phi(0,0)|
\le
L_\phi|x'|
<
L_*\rho_*.
\]
By the definition of $\rho_*$, $L_*\rho_*=\frac{\mu r_*}{8}$, and 
\[
s-\frac{\mu r_*}{8}<\phi(x',0)<s+\frac{\mu r_*}{8}.
\]
Using $\partial_n\phi\ge\frac{\mu}{2}$ in $\mathcal C_*$, we obtain
\begin{align*}
\phi\left(x',\frac{r_*}{2}\right)&=\phi(x',0)+\int_0^{r_*/2}\partial_n\phi(x',t)\,dt\\
&\ge\phi(x',0)+\frac{\mu}{2}\frac{r_*}{2}>s-\frac{\mu r_*}{8}+\frac{\mu r_*}{4}=s+\frac{\mu r_*}{8}>s.
\end{align*}
Similarly, $\phi\left(x',-\frac{r_*}{2}\right)<s$. Thus together
\[
\phi\left(x',-\frac{r_*}{2}\right)<s<\phi\left(x',\frac{r_*}{2}\right).
\]
For $x'\in B'_{\rho_*}$, consider $t\to\phi(x',t)$. This function is continuous, and $\partial_n\phi(x',t)\ge\frac{\mu}{2}>0$. It is thus strictly increasing on $\left(-\frac{r_*}{2},\frac{r_*}{2}\right)$. Then there is a unique $\psi(x')\in\left(-\frac{r_*}{2},\frac{r_*}{2}\right)$ such that $\phi(x',\psi(x'))=s$. Since $\partial_n\phi(x',\psi(x'))\ge\frac{\mu}{2}>0$, the implicit-function theorem gives $\psi\in C^{1,\vartheta}(B'_{\rho_*})$. We have $\phi(x',x_n)>s$ if only if $x_n>\psi(x')$. Thus $\Gamma_s\cap\mathcal C_*=\{(x',x_n)\in\mathcal C_*:x_n=\psi(x')\}$ and
\[
\Omega_s\cap\mathcal C_*=\{(x',x_n)\in\mathcal C_*:x_n>\psi(x')\}.
\]
Because the origin corresponds to $x_0\in\Gamma_s$, $\phi(0,0)=s$. By uniqueness, $\psi(0)=0$. Differentiating $\phi(x',\psi(x'))=s$ gives $\nabla'\phi+\partial_n\phi\,\nabla'\psi=0$.
Therefore,
\[
\nabla'\psi(x')
=-\frac{\nabla'\phi}{\partial_n\phi}\bigl(x',\psi(x')\bigr).
\]
At $x'=0$, $\nabla'\phi(0)=0$, and $\nabla'\psi(0)=0$. Further,
\[
|\nabla'\psi(x')|\le\frac{|\nabla'\phi|}{\partial_n\phi}\le\frac{\mu/2}{\mu/2}=1.
\]
Thus $\psi(0)=0$, $\nabla'\psi(0)=0$, $\|\nabla'\psi\|_{L^\infty(B'_{\rho_*})}\le1$. Define $\eta:=-\frac{\nabla'\phi}{\partial_n\phi}$. 
For $x,y\in\mathcal C_*$,
\begin{align*}
|\eta(x)-\eta(y)|
&\le\frac{|\nabla'\phi(x)-\nabla'\phi(y)|}{\partial_n\phi(x)}+|\nabla'\phi(y)|\frac{|\partial_n\phi(x)-\partial_n\phi(y)|}{
\partial_n\phi(x)\partial_n\phi(y)}.
\end{align*}
Using $\partial_n\phi(x)$, $\partial_n\phi(y)\ge\frac{\mu}{2}$, $|\nabla'\phi(y)|\le\frac{\mu}{2}$ with
\[
|\nabla'\phi(x)-\nabla'\phi(y)|
\le
H_\phi|x-y|^\vartheta
\]
and $|\partial_n\phi(x)-\partial_n\phi(y)|
\le H_\phi|x-y|^\vartheta$, we obtain $|\eta (x)-\eta(y)|\le\frac{4H_\phi}{\mu}|x-y|^\vartheta$. Thus $[\eta]_{C^{0,\vartheta}(\mathcal C_*)}
\le\frac{4H_\phi}{\mu}$. We now define $G(x'):=(x',\psi(x'))$. Since $\|\nabla'\psi\|_{L^\infty}\le1$, we have
\begin{align*}
|G(x')-G(y')|^2&=|x'-y'|^2+|\psi(x')-\psi(y')|^2\le 2|x'-y'|^2.
\end{align*}
Therefore, $|G(x')-G(y')|\le\sqrt2\,|x'-y'|$ and since $\nabla'\psi=\eta\circ G$, we obtain
\begin{align*}
|\nabla'\psi(x')-\nabla'\psi(y')|&\le\frac{4H_\phi}{\mu}
|G(x')-G(y')|^\vartheta\le
\frac{4H_\phi}{\mu}
(\sqrt2)^\vartheta
|x'-y'|^\vartheta.
\end{align*}
Thus,
\[
[\nabla'\psi]_{C^{0,\vartheta}(B'_{\rho_*})}
\le
\frac{
2^{2+\vartheta/2}H_\phi
}{
\mu
}
=
M_*.
\]

Since $\phi\in C(\overline{\Omega})$, the set $\{x\in\overline{\Omega}:\phi(x)\ge s\}$ is compact and since $\phi=0$ on $\partial\Omega$ and $s>0$, it is disjoint from $\partial\Omega$. Thus clearly $[\phi\ge s]\Subset\Omega$. Now let $x\in\partial\Omega_s$. Because $\overline{\Omega_s}\Subset\Omega$, there exist sequences $x_j\in\Omega_s$, $y_j\in\Omega\setminus\Omega_s$, both converging to $x$. Thus $\phi(x_j)>s$, $\phi(y_j)\le s$. By continuity, $\phi(x)\ge s$ and $\phi(x)\le s$. Therefore $\phi(x)=s$, and $\partial\Omega_s\subset\Gamma_s$. On the other hand, let $x\in\Gamma_s$. The local graph representation contains points with $x_n>\psi(x')$ and points with $x_n<\psi(x')$. The first belonging to $\Omega_s$, while the second does not. Thus every neighborhood of $x$ intersects both $\Omega_s$ and its complement, and therefore $x\in\partial\Omega_s$. Thus, $\partial\Omega_s=\Gamma_s$. 

The set $\Gamma_s$ is compact because it is a closed subset of the
compact set $[\phi\ge s]\Subset\Omega$. The inward unit normal is $\frac{\nabla\phi}{|\nabla\phi|}$, implying the outward unit normal being $\nu_s=-\frac{\nabla\phi}{|\nabla\phi|}$. 

It remains to prove (d). We choose a maximal separated set $\{x_1,\ldots,x_{N_s}\}\subset\Gamma_s$ such that $|x_i-x_j|\ge\frac{\rho_*}{4}$ for $i\ne j$. By the maximality, $\Gamma_s
\subset \bigcup_{j=1}^{N_s} B_{\rho_*/4}(x_j)$. The balls $B_{\rho_*/8}(x_j)$ are pairwise disjoint. If two of them intersected, then $|x_i-x_j|<\frac{\rho_*}{8}+\frac{\rho_*}{8}=\frac{\rho_*}{4}$, which is a contradiction. We take $R>0$ such that $\Omega\subset B_R(0)$. Then $B_{\rho_*/8}(x_j)\subset B_{R+\rho_*/8}(0)$, and $N_s|B_{\rho_*/8}|\le|B_{R+\rho_*/8}|$. Therefore,
\[
N_s
\le
\left(
\frac{R+\rho_*/8}{\rho_*/8}
\right)^n.
\]
The right-hand side is independent of $s$. Thus there exists
$N_*\in\mathbb N$ such that $N_s\le N_*$ for all $s\in[\tau_1,\sigma_1]$. 
In the local coordinates corresponding to the chart centered at $x_j$,
let $x\in B_{\rho_*/2}(0)$. Then $|x'|<\frac{\rho_*}{2}<\rho_*$ and $|x_n|<\frac{\rho_*}{2}\le\frac{r_*}{16}<\frac{r_*}{2}$, because $\rho_*\le\frac{r_*}{8}$. Therefore
\[
B_{\rho_*/2}(0)\subset B'_{\rho_*}\times\left(-\frac{r_*}{2},
\frac{r_*}{2}\right)=\mathcal C_*.
\]
After translating and rotating back, $B_{\rho_*/2}(x_j)$ lies in the corresponding cylinder. Finally, if
\[
x\in B_{\rho_*/2}(x_{j_1}) \cap\cdots\cap B_{\rho_*/2}(x_{j_m}), 
\]
then $|x_{j_\ell}-x|<\frac{\rho_*}{2}$. For every $y\in B_{\rho_*/8}(x_{j_\ell})$, we have
\[
|y-x|\le|y-x_{j_\ell}|+|x_{j_\ell}-x|<\frac{\rho_*}{8}+ \frac{\rho_*}{2}=\frac{5\rho_*}{8}.
\]
Thus the pairwise disjoint balls $B_{\rho_*/8}(x_{j_\ell})$ are all contained in $B_{5\rho_*/8}(x)$. Thus, $m|B_{\rho_*/8}|\le |B_{5\rho_*/8}|$, and therefore $m\le5^n$. Thus $\sum_{j=1}^{N_s}\chi_{B_{\rho_*/2}(x_j)}\le5^n$ and this bound is independent of $s$.
\end{proof}
In what follows, constants may depend on $\rho_*$, $r_*$, $M_*$, $N_*$ and on the data in
Assumption \ref{AS5.1}, but not on $s\in[\tau_1,\sigma_1]$.
\smallskip

For nonempty $\Omega_s$, let $\gamma_s:H^1(\Omega_s)\to L^2(\Gamma_s)$ denote the trace operator. We do not assume that $\Omega_s$ is connected in the following result.
\begin{lemma}
\label{LE5.2}
There exists a constant $\widehat C>0$, independent of $s$, such that for every $s\in[\tau_1,\sigma_1]$ and every $z\in H^1(\Omega_s)$ satisfying
\[
\int_{\Omega_s}\left(A\nabla z\cdot\nabla\varphi+az\varphi\right)\,dx=0\quad\forall\varphi\in H_0^1(\Omega_s),
\]
it holds
\[
\|z\|_{L^2(\Omega_s)}\le\widehat C\|\gamma_s z\|_{L^2(\Gamma_s)}.
\]
\end{lemma}
\begin{proof}
We take $s\in[\tau_1,\sigma_1]$. If $\Omega_s=\varnothing$, there is nothing to prove. We assume thus that $\Omega_s\ne\varnothing$.

From Lemma \ref{LE5.1} we infer that $\Omega_s\Subset\Omega$, $\Gamma_s=\partial\Omega_s$ and $\Gamma_s$ is Lipschitz. Furhter, there exist $\rho_*,r_*>0$, $N_*\in\mathbb N$, independent of $s$, with the following properties. For every $x_0\in\Gamma_s$, after a change of coordinates, $\Gamma_s$ is the graph of a function $\psi$ in
\[
B'_{\rho_*}\times\left(-\frac{r_*}{2},\frac{r_*}{2}\right),
\]
with $\|\nabla'\psi\|_{L^\infty(B'_{\rho_*})}\le1$ and $\Gamma_s$ is covered by at most $N_*$ such cylinders. Using the smaller cylinders from the proof of Lemma \ref{LE5.1}, the number of intersections of the cylinders is bounded by $5^n$. A subordinate partition of unity can then be chosen such that its first derivatives are bounded by $\frac{C}{\rho_*}$, with $C$ independent of $s$. Let $v\in H_0^1(\Omega_s)$. We extend $v$ to $\Omega$ by zero
\[
\widetilde v(x):=\begin{cases}v(x),&x\in\Omega_s,\\
0,&x\in\Omega\setminus\Omega_s.
\end{cases}
\]
Then $\widetilde v\in H_0^1(\Omega)$, and $\|v\|_{L^2(\Omega_s)}=\|\widetilde v\|_{L^2(\Omega)}$, $\|\nabla v\|_{L^2(\Omega_s)}=\|\nabla\widetilde v\|_{L^2(\Omega)}$. The Poincar\'e inequality on $\Omega$ gives $\|v\|_{L^2(\Omega_s)}=\|\widetilde v\|_{L^2(\Omega)}\le C_\Omega\|\nabla\widetilde v\|_{L^2(\Omega)}=C_\Omega \|\nabla v\|_{L^2(\Omega_s)}$, where $C_\Omega$ is independent of $s$. Thus,
\[
\|v\|_{H^1(\Omega_s)}^2\le(1+C_\Omega^2)\|\nabla v\|_{L^2(\Omega_s)}^2.
\]

We introduce the bilinear form
\[
a_{\Omega_s}(u,v):=\int_{\Omega_s}\left(A\nabla u\cdot\nabla v+auv\right)\,dx.
\]
Since $A$ is uniformly elliptic and $a\ge0$,
\[
a_{\Omega_s}(v,v)\ge\lambda_A\|\nabla v\|_{L^2(\Omega_s)}^2\ge a_0\|v\|_{H^1(\Omega_s)}^2 \quad\forall v\in H_0^1(\Omega_s),
\]
where $a_0:=\frac{\lambda_A}{1+C_\Omega^2}$. Let $\widetilde C_{A,a}:=2\max\left\{\|A\|_{L^\infty(\Omega)},C_\Omega^2\|a\|_{L^\infty(\Omega)}\right\}$. Then $|a_{\Omega_s}(u,v)|\le\widetilde C_{A,a}\|u\|_{H^1(\Omega_s)}\|v\|_{H^1(\Omega_s)}$ for $u,v\in H_0^1(\Omega_s)$.

Given $z\in L^2(\Omega_s)$, we consider the weak PDE
\[
Lw=z \quad\text{in }\Omega_s, \quad w=0 \quad\text{on }\Gamma_s,
\]
that is,
\[
a_{\Omega_s}(w,\varphi)=\int_{\Omega_s}z\varphi\,dx
\quad\forall\varphi\in H_0^1(\Omega_s).
\]
The right-hand side satisfies
\[
\left|\int_{\Omega_s}z\varphi\,dx\right|\le C_\Omega\|z\|_{L^2(\Omega_s)}\|\varphi\|_{H^1(\Omega_s)}.
\]
From Lax-Milgram we obtain a unique $w\in H_0^1(\Omega_s)$. If $v\in H_0^1(\Omega_s)$ satisfies
$Lv=0$, then $v=0$. Thus the zero is not a Dirichlet eigenvalue of $L$. Since $A$ is symmetric and $a$ is real-valued, the operator is self-adjoint. Assume that $\Omega_s$ is connected. For $h\in H^1(\Gamma_s)$, the  result \cite[Lemma 1.1, Chapter 5]{Necas2012} gives a linear operator
\[
E_s:H^1(\Gamma_s)\to H^1(\Omega_s)
\]
such that $\gamma_s(E_sh)=h$ and
\[
\|E_sh\|_{H^1(\Omega_s)}\le C_E \|h\|_{H^1(\Gamma_s)}.
\]
Following the proof of \cite[Lemma 1.1, Chapter 5]{Necas2012}, the constant $C_E$
depends on the number of boundary charts, the bounds related to the graph presentation, and the bounds related to the partition of unity. By Lemma \ref{LE5.1}, these quantities are bounded uniformly in $s\in[\tau_1,\sigma_1]$. Thus $C_E$ can be chosen independently of $s$. If $\Omega_s$ consists of finitely many disjoint connected components, we write
\[
\Omega_s=\bigcup_\ell V_\ell,\quad h_\ell:= h|_{\partial V_\ell},
\]
and define $E_sh|_{V_\ell}:= E_{s,\ell}h_\ell$. We can choose the same uniform constant on every component, therefore 
\[
\|E_sh\|_{H^1(\Omega_s)}^2 =\sum_\ell \|E_{s,\ell}h_\ell\|_{H^1(V_\ell)}^2 \le C_E^2 \sum_\ell \|h_\ell\|_{H^1(\partial V_\ell)}^2=C_E^2\|h\|_{H^1(\Gamma_s)}^2.
\]

Let now $f\in L^2(\Omega_s)$, $h\in H^1(\Gamma_s)$, and consider
\[
Lv=f\quad\text{in }\Omega_s,\quad\gamma_s v=h\quad\text{on }\Gamma_s.
\]
We decompose $v=E_sh+v_0$, $v_0\in H_0^1(\Omega_s)$. Then $v$ solves the Dirichlet problem if and only if
\[
a_{\Omega_s}(v_0,\varphi)=\int_{\Omega_s} f\varphi\,dx-a_{\Omega_s}(E_sh,\varphi)\quad\forall\varphi\in H_0^1(\Omega_s).
\]
For every $\varphi\in H_0^1(\Omega_s)$,
\[
\left|\int_{\Omega_s}f\varphi\,dx\right|\le C_\Omega\|f\|_{L^2(\Omega_s)}\|\varphi\|_{H^1(\Omega_s)}.
\]
Let $C_{A,a}:=2\max\left\{\|A\|_{L^\infty(\Omega)},\|a\|_{L^\infty(\Omega)}\right\}$. Then
\[
|a_{\Omega_s}(E_sh,\varphi)|\le C_{A,a}
\|E_sh\|_{H^1(\Omega_s)} \|\varphi\|_{H^1(\Omega_s)}
\]
and therefore
\[
|a_{\Omega_s}(E_sh,\varphi)|
\le C_{A,a}C_E\|h\|_{H^1(\Gamma_s)} \|\varphi\|_{H^1(\Omega_s)}.
\]
Hence
\begin{align*}
\left|
\int_{\Omega_s}f\varphi\,dx- a_{\Omega_s}(E_sh,\varphi)\right|
&\le\left(C_\Omega\|f\|_{L^2(\Omega_s)}+
C_{A,a}C_E\|h\|_{H^1(\Gamma_s)}\right)
\times
\|\varphi\|_{H^1(\Omega_s)}.
\end{align*}
Using $a_{\Omega_s}(\varphi,\varphi)\ge a_0 \|\varphi\|_{H^1(\Omega_s)}^2$, Lax-Milgram gives a unique $v_0\in H_0^1(\Omega_s)$ satisfying
\[
\|v_0\|_{H^1(\Omega_s)}\le\frac{1}{a_0}\left(
C_\Omega\|f\|_{L^2(\Omega_s)}+C_{A,a}C_E\|h\|_{H^1(\Gamma_s)}\right).
\]
Since $v=E_sh+v_0$, we obtain
\[
\|v\|_{H^1(\Omega_s)}\le\frac{C_\Omega}{a_0}
\|f\|_{L^2(\Omega_s)}+C_E\left(1+\frac{C_{A,a}}{a_0}\right)\|h\|_{H^1(\Gamma_s)}.
\]
Thus $\|v\|_{H^1(\Omega_s)}\le C_1\left(\|f\|_{L^2(\Omega_s)}+\|h\|_{H^1(\Gamma_s)}\right)$, where
\[
C_1:=\max\left\{\frac{C_\Omega}{a_0},C_E\left(1+\frac{C_{A,a}}{a_0}\right)
\right\}
\]
and $C_1$ is independent of $s$. We now consider the solution $w\in H_0^1(\Omega_s)$ of
\[
a_{\Omega_s}(w,\varphi)=\int_{\Omega_s}z\varphi\,dx\quad
\forall\varphi\in H_0^1(\Omega_s).
\]
The conormal-derivative estimate from
\cite[Section 5.1.2, Lemma 1.4]{Necas2012}
gives a conormal derivative $g_s:=\partial_{\nu_A}w\in L^2(\Gamma_s)$ satisfying $\|g_s\|_{L^2(\Gamma_s)}\le C_\nu \|z\|_{L^2(\Omega_s)}$. The proof shows that $C_\nu$ depends on the elliptic operator, the radii of the boundary
cylinders, the gradient bounds for the graph functions, the number and amount of intersections of the cylinders, the derivatives for the partition of unity, and the norm of the Dirichlet operator. All of these are bounded independently of $s$.
If $\Omega_s$ is disconnected, write
\[
\Omega_s=\bigcup_{\ell=1}^{m_s} V_{s,\ell}.
\]
The number $m_s$ is finite and bounded by the boundary covering from Lemma \ref{LE5.1}. We define $z_{s,\ell}:=z|_{V_{s,\ell}}$, $w_{s,\ell}:=w|_{V_{s,\ell}}$. On each component, $Lw_{s,\ell}=z_{s,\ell}$, $w_{s,\ell}=0$ on $\partial V_{s,\ell}$. On each component, the needed estimates are satisfied with the same bounds. Therefore, $g_{s,\ell}:=\partial_{\nu_A}w_{s,\ell}\in L^2(\partial V_{s,\ell})$ and $\|g_{s,\ell}\|_{L^2(\partial V_{s,\ell})}\le
C_\nu\|z_{s,\ell}\|_{L^2(V_{s,\ell})}$. Then, defining $g_s$ component-wise, we thus obtain
\begin{align*}
\|g_s\|_{L^2(\Gamma_s)}^2=\sum_{\ell=1}^{m_s}\|g_{s,\ell}\|_{L^2(\partial V_{s,\ell})}^2&\le
C_\nu^2\sum_{\ell=1}^{m_s}\|z_{s,\ell}\|_{L^2(V_{s,\ell})}^2=C_\nu^2\|z\|_{L^2(\Omega_s)}^2.
\end{align*}

Since $z\in H^1(\Omega_s)$ satisfies $a_{\Omega_s}(z,\varphi)=0$  for all $\varphi\in H_0^1(\Omega_s)$, the Green identity
\cite[Section 5.1.3, equation (5.21)]{Necas2012}, applied to $z$ and
$w$, gives
\[
\|z\|_{L^2(\Omega_s)}^2=-\int_{\Gamma_s}(\gamma_s z)g_s\,dS+a_{\Omega_s}(z,w).
\]
Since $w\in H_0^1(\Omega_s)$, the weak harmonicity of $z$ gives $a_{\Omega_s}(z,w)=0$. Thus
\[
\|z\|_{L^2(\Omega_s)}^2=-\int_{\Gamma_s}(\gamma_s z)g_s\,dS.
\]
By Cauchy-Schwarz,
\[
\|z\|_{L^2(\Omega_s)}^2\le \|\gamma_s z\|_{L^2(\Gamma_s)}\|g_s\|_{L^2(\Gamma_s)}\le C_\nu \|\gamma_s z\|_{L^2(\Gamma_s)}\|z\|_{L^2(\Omega_s)}.
\]
Division by $\|z\|_{L^2(\Omega_s)}$ gives $\|z\|_{L^2(\Omega_s)}\le C_\nu\|\gamma_s z\|_{L^2(\Gamma_s)}$. 
\end{proof}

\subsection{Main result of the section}
We can now prove the main result of this section.
\begin{theorem}
\label{TH5.1}
Let Assumption \ref{AS5.1} hold and let $\tau_1\le\tau<\sigma_1$. Let $z\in H^1(\Omega_\tau)$ satisfy $Lz=0\quad\text{weakly in }\Omega_\tau$. The function
\[
\theta(s):=\int_{\Omega_s}|z(x)|^2\,dx,\quad s\in[\tau,\sigma_1].
\]
is absolutely continuous on $[\tau,\sigma_1]$ and
\begin{equation}
\label{EQ5.3}
-\theta'(s)\ge\frac{1}{C_0}\theta(s)\quad\text{for a.e. }s\in(\tau,\sigma_1),
\end{equation}
where $C_0:=\widehat C^{\,2}L_*$, and $\widehat C$ is the constant from
Lemma \ref{LE5.2}. For every $\tau\le s_1<s_2\le\sigma_1$, it holds
\[
\|z\|_{L^2(\Omega_{s_2})}^2\le\exp\left(-\frac{s_2-s_1}{C_0}\right) \|z\|_{L^2(\Omega_{s_1})}^2.
\]
The constant $C_0$ is independent of $z$ and of $s_1,s_2$.
\end{theorem}

\begin{proof}
Let $\tau\le s_1<s_2\le\sigma_1$. Since $|\nabla\phi|\ge\mu>0$ on $[\tau_1\le\phi\le\sigma_1]$, the coarea formula gives
\begin{align}
\theta(s_1)-\theta(s_2)&=\int_{[s_1<\phi<s_2]}|z(x)|^2\,dx=\int_{s_1}^{s_2}\int_{\Gamma_s}\frac{|z(x)|^2}{|\nabla\phi(x)|}\,d\mathcal H^{n-1}(x)\,ds.
\label{EQ5.4}
\end{align}
It follows from \eqref{EQ5.4} that $\theta$ is absolutely
continuous on $[\tau,\sigma_1]$ and that, for almost all
$s\in(\tau,\sigma_1)$,
\[
-\theta'(s)=\int_{\Gamma_s}\frac{|z|^2}{|\nabla\phi|}\,d\mathcal H^{n-1}.
\]
Since $|\nabla\phi|\le L_\phi\le L_*$, we obtain
\begin{equation}
\label{EQ5.5}
-\theta'(s)\ge\frac{1}{L_*}\|\gamma_s z\|_{L^2(\Gamma_s)}^2.
\end{equation}
The function $z$ is $L$-harmonic in $\Omega_s$ for $s\ge\tau$. Therefore, Lemma \ref{LE5.2} gives
\[
\theta(s)=\|z\|_{L^2(\Omega_s)}^2\le\widehat C^{\,2}\|\gamma_s z\|_{L^2(\Gamma_s)}^2,
\]
and $\|\gamma_s z\|_{L^2(\Gamma_s)}^2\ge\frac{1}{\widehat C^{\,2}}\theta(s)$. Combining with
\eqref{EQ5.5}, we obtain $$-\theta'(s)\ge\frac{1}{\widehat C^{\,2}L_*}\theta(s)=\frac{1}{C_0}\theta(s)$$ for almost all $s\in(\tau,\sigma_1)$. This proves \eqref{EQ5.3}. Further, for a. a. $s\in(\tau,\sigma_1)$, 
\[
\theta'(s)+\frac{1}{C_0}\theta(s)\le0.
\]
Therefore, $\frac{d}{ds}\left(e^{s/C_0}\theta(s)\right)\le0$. Then,
\[
e^{s_2/C_0}\theta(s_2)\le e^{s_1/C_0}\theta(s_1),
\]
and $\theta(s_2)\le\exp\left(-\frac{s_2-s_1}{C_0}\right)\theta(s_1)$. 
\end{proof}

\begin{remark}
Let $\tau_1\le\tau<k\le\sigma_1$. Applying Theorem \ref{TH5.1} with $s_1=\tau$ and $s_2=k$ gives
\[
\|z\|_{L^2(\Omega_k)}^2\le\exp\left(-\frac{k-\tau}{C_0}\right)\|z\|_{L^2(\Omega_\tau)}^2.
\]
Since $\|z\|_{L^2(\Omega_\tau)}^2=\|z\|_{L^2(\Omega_k)}^2+\|z\|_{L^2([\tau<\phi\le k])}^2$, 
\[
\|z\|_{L^2(\Omega_k)}^2\le\frac{1}{\exp((k-\tau)/C_0)-1}\,\|z\|_{L^2([\tau<\phi\le k])}^2.
\]
\end{remark}

\section{Application to the optimal control problem}
\label{sec6}
We now apply the results from Section \ref{sec5} to the linearized states associated with the elliptic optimal control problem. 
\begin{assumption}
\label{AS6.1}
There exist constants $0<\tau_1<\sigma_1$ and $\mu>0$ such that at a reference point, the adjoint $\bar p$ satisfies
\[
\mu\le|\nabla\bar p(x)|\quad\text{for every }x\in\overline{\Omega}\text{ with }\tau_1\le|\bar p(x)|\le\sigma_1.
\]
\end{assumption}
From \eqref{EQ2.8} we recall that for $\tau>0$,
\[
K_{\bar p}^{\tau}=\left\{h\in L^2(\Omega):h=0\quad\text{a.e. on }[|\bar p|>\tau]\right\},
\]
and for $s>0$, we remind that $\Omega_s=[|\bar p|>s]$, $\Gamma_s=[|\bar p|=s]$. 
\begin{proposition}
\label{PR6.1}
Let Assumptions \ref{A2.1} and \ref{AS6.1} hold. Let $\tau_1\le\tau\le s_1<s_2\le\sigma_1$ and let $h\in K_{\bar p}^{\tau}$. Then
\begin{equation*}
\|z_{\bar u,h}\|_{L^2(\Omega_{s_2})}^2 \le \exp\left(-\frac{s_2-s_1}{C_0}\right)\|z_{\bar u,h}\|_{L^2(\Omega_{s_1})}^2,
\end{equation*}
where $C_0=\widehat C^{\,2}L_*$ is the constant from
Theorem \ref{TH5.1} corresponding to $\phi=|\bar p|$, $a=f_y(\cdot,\bar y)$. 
\end{proposition}

\begin{proof}
We apply Theorem \ref{TH5.1} with $\phi:=|\bar p|$, $a:=f_y(\cdot,\bar y)$, $z:=z_{\bar u,h}$. By Assumption \ref{A2.1}, $a=f_y(\cdot,\bar y)\in L^\infty(\Omega)$ and $a\ge0$ a.e. in $\Omega$. The matrix $A$ satisfies Assumption \ref{AS5.1}. By \eqref{EQ2.3}, $\bar p \in C^{1,\vartheta}(\overline{\Omega})$, $|\bar p|\in W^{1,\infty}(\Omega)\cap C(\overline{\Omega})$ and $|\bar p|=0$ on $\partial\Omega$.  On the set $[0<|\bar p|]$, the function $|\bar p|$ belongs locally to $C^{1,\vartheta}$ and
\begin{equation}
\label{EQ6.1}
\nabla|\bar p|=\operatorname{sgn}(\bar p)\nabla\bar p,\quad
|\nabla|\bar p||=|\nabla\bar p|.
\end{equation}
The set $U_{\tau_1}:=\left\{x\in\Omega:|\bar p(x)|>\frac{\tau_1}{4}\right\}$ is a neighborhood of $\left\{x\in\Omega:|\bar p(x)|\ge\frac{\tau_1}{2}\right\}$ on which $|\bar p|$ belongs to $C^{1,\vartheta}$. If $x,y\in U_{\tau_1}$ and $\bar p(x)$ and $\bar p(y)$ have the same sign, then
\[
|\nabla|\bar p|(x)-\nabla|\bar p|(y)|\le H_{\bar p}|x-y|^\vartheta.
\]
If they have opposite signs, then $|\bar p(x)-\bar p(y)|>\frac{\tau_1}{2}$, and by the Lipschitz continuity of $\bar p$, $|x-y|>\frac{\tau_1}{2L_{\bar p}}$. Then
\[
|\nabla|\bar p|(x)-\nabla|\bar p|(y)|\le 2L_{\bar p}\le 2L_{\bar p}
\left(\frac{2L_{\bar p}}{\tau_1}\right)^\vartheta |x-y|^\vartheta.
\]
Thus the $C^{1,\vartheta}$-bound in Assumption \ref{AS5.1} holds on the neighborhood $U_{\tau_1}$. Furthermore, Assumption \ref{AS6.1} and
\eqref{EQ6.1} give $|\nabla|\bar p||=|\nabla\bar p|\ge\mu$ on $[\tau_1\le|\bar p|\le\sigma_1]$. Thus, the function $\phi=|\bar p|$ satisfies Assumption \ref{AS5.1}. Since $h\in K_{\bar p}^{\tau}$, we have $h=0$ a.e. on $[|\bar p|>\tau]=\Omega_\tau$. The linearized state satisfies
\[
-\operatorname{div}(A\nabla z_{\bar u,h})+f_y(\cdot,\bar y)z_{\bar u,h}=h \quad\text{in }\Omega.
\]
Therefore,
\[
-\operatorname{div}(A\nabla z_{\bar u,h})+f_y(\cdot,\bar y)z_{\bar u,h}=0\quad\text{weakly in }\Omega_\tau.
\]
The claim follows from Theorem \ref{TH5.1}.
\end{proof}

\begin{corollary}
\label{CO6.1}
Let Assumptions \ref{A2.1} and \ref{AS6.1} hold and let
$\bar u\in\mathcal U$ be stationary. Let $\tau_1\le\tau<k\le\sigma_1$, $0<k<\frac1m$. Then, for every $h\in K_{\bar p}^{\tau}$, \begin{equation}
\label{EQ6.2}
J''(\bar u)[h,h]\ge{\bar \kappa}(\tau,k)\|z_{\bar u,h}\|_{L^2(\Omega)}^2,
\end{equation}
where
\begin{equation}
\label{EQ6.3}
\begin{aligned}
{\bar \kappa}(\tau,k)&:=\delta(k)\left[1-\exp\left(-\frac{k-\tau}{C_0}\right)\right]-M(k)\exp\left(-\frac{k-\tau}{C_0}\right)\\
&=\delta(k)-\bigl(\delta(k)+M(k)\bigr)\exp\left(-\frac{k-\tau}{C_0}\right).
\end{aligned}
\end{equation}
The constant ${\bar \kappa}(\tau,k)$ is positive if and only if
\begin{equation}
\label{EQ6.4}
M(k)<\delta(k)\left[\exp\left(\frac{k-\tau}{C_0}\right)-1\right].
\end{equation}
\end{corollary}

\begin{proof}
By Lemma \ref{LE5.1}, the level set $\Gamma_k=[|\bar p|=k]$ is an $(n-1)$-dimensional $C^{1,\vartheta}$-hypersurface, thus has zero $n$-dimensional Lebesgue measure and $\|z\|_{L^2([k\le|\bar p|])}^2=\|z\|_{L^2(\Omega_k)}^2$. Proposition \ref{PR6.1}, applied with
$s_1=\tau$ and $s_2=k$, gives
\begin{align}
\|z\|_{L^2([k\le|\bar p|])}^2&=\|z\|_{L^2(\Omega_k)}^2\nonumber\\
&\le\exp\left(-\frac{k-\tau}{C_0}\right)\|z\|_{L^2(\Omega_\tau)}^2\nonumber\le
\exp\left(-\frac{k-\tau}{C_0}\right)\|z\|_{L^2(\Omega)}^2.
\end{align}
By \eqref{EQ2.9}, $c\ge\delta(k)$ on $[|\bar p|<k]$, and, by the definition of $M(k)$, $c\ge-M(k)$ on $[k\le|\bar p|]$. Therefore,
\begin{align*}
J''(\bar u)[h,h]&=\int_\Omega cz^2\,dx\ge \delta(k)\|z\|_{L^2([|\bar p|<k])}^2-M(k)\|z\|_{L^2([k\le|\bar p|])}^2\\
&=\delta(k)\|z\|_{L^2(\Omega)}^2-\bigl(\delta(k)+M(k)\bigr)\|z\|_{L^2([k\le|\bar p|])}^2.
\end{align*}
We obtain
\[
J''(\bar u)[h,h]\ge\left[\delta(k)-\bigl(\delta(k)+M(k)\bigr)\exp\left(-\frac{k-\tau}{C_0}\right)\right]\|z\|_{L^2(\Omega)}^2.
\]
This proves \eqref{EQ6.2} and \eqref{EQ6.3}.
Finally, ${\bar \kappa}(\tau,k)>0$ is equivalent to
\[
\delta(k)>\bigl(\delta(k)+M(k)\bigr)\exp\left(-\frac{k-\tau}{C_0}\right).
\]
which gives equivalently $M(k)<\delta(k)\left[\exp\left(\frac{k-\tau}{C_0}\right)-1\right]$.
\end{proof}

\begin{corollary}
\label{CO6.2}
Let the assumptions of Corollary \ref{CO6.1} hold and suppose that \eqref{EQ6.4} is satisfied. Then
\[
J''(\bar u)[h,h]\ge {\bar \kappa}(\tau,k)\|z_{\bar u,h}\|_{L^2(\Omega)}^2\quad \forall h\in C_{\bar u}^{\tau},
\]
\end{corollary}
\begin{proof}
The claim follows from $C_{\bar u}^{\tau}\subset K_{\bar p}^{\tau}$. 
\end{proof}

\begin{corollary}
Let Assumptions \ref{A2.1} and \ref{AS6.1} hold and let $\bar u\in\mathcal U$ be stationary. Suppose that there exist $\tau_1\le\tau<k\le\sigma_1$, $0<k<\frac1m$, such that
\begin{equation}
\label{EQ6.5}
M(k)<\delta(k)\left[\exp\left(\frac{k-\tau}{C_0}\right)-1\right].
\end{equation}
Then there exist constants $\rho>0$ and $\gamma>0$ such that
\begin{equation}
\label{EQ6.6}
J(u)\ge J(\bar u)+\frac{\gamma}{2}\|y_u-\bar y\|_{L^2(\Omega)}^2
\end{equation}
for every $u\in\mathcal U$ with $\|y_u-\bar y\|_{L^\infty(\Omega)}\le\rho$.
\end{corollary}
\begin{proof}
By Corollary \ref{CO6.2}, condition \eqref{EQ6.5} implies the existence of ${\bar \kappa}(\tau,k)>0$ such that
\[
J''(\bar u)[h,h]\ge{\bar \kappa}(\tau,k)\|z_{\bar u,h}\|_{L^2(\Omega)}^2\quad\forall h\in C_{\bar u}^{\tau}.
\]
This is the second-order sufficient condition on the extended critical cone. The  estimate \eqref{EQ6.6} follows from the second-order sufficiency result for unregularized semilinear elliptic control problems, \cite{Casas2012, CasasMateos2020}.
\end{proof}

\section{Stability of the sufficient second-order condition}
\label{sec7}
In this section, we show that the sufficient second-order optimality
condition obtained is stable under perturbations. Let $\rho$ be a parameter and consider the objective functional
\[
J_\rho(u):=\frac12\int_\Omega(y_u-y_{d,\rho})^2\,dx,
\]
where $y_{d,\rho}\in L^r(\Omega)$, $r>n$, is a perturbed tracking target. We consider a stationary point $\bar u_\rho\in\mathcal U$ for $J_\rho$ and set $y_\rho:=y_{\bar u_\rho}$. The corresponding adjoint state $p_\rho$ satisfies
\[
-\operatorname{div}(A\nabla p_\rho)+f_y(\cdot,y_\rho)p_\rho=y_\rho-y_{d,\rho}\quad\text{in }\Omega, \quad p_\rho=0 \quad\text{on }\Gamma.
\]
For $b_\rho:=f_y(\cdot,y_\rho)$, the linearized state $z_{\rho,h}$ at $\bar u_\rho$ satisfies
\[
-\operatorname{div}(A\nabla z_{\rho,h})+ b_\rho z_{\rho,h}=h \quad\text{in }\Omega, \quad z_{\rho,h}=0 \quad\text{on }\Gamma.
\]
We define $c_\rho := 1-p_\rho f_{yy}(\cdot,y_\rho)\in L^\infty(\Omega)$. The second variation of the perturbed reduced functional is
\[
J_\rho''(\bar u_\rho)[h,h]=\int_\Omega c_\rho z_{\rho,h}^2\,dx.
\]
For $\tau>0$, the extended critical cone of the perturbed problem is
\[
C_{\bar u_\rho}^{\tau}:=\left\{h\in T_{\mathcal U}(\bar u_\rho):h=0
\quad\text{a.e. on }[|p_\rho|>\tau]\right\}.
\]

Let $\tau_1\le\widetilde\tau<k\le\sigma_1$, $0<k<\frac1m$. If $M(k)<\delta(k)\left[\exp\left(\frac{k-\widetilde\tau}{C_0}\right)-1\right]$, then Corollary \ref{CO6.1} gives
\begin{equation}
\label{E5.1}
J''(\bar u)[h,h]\ge{\bar \kappa}(\widetilde\tau,k)\|z_{\bar u,h}\|_{L^2(\Omega)}^2\quad\forall h\in K_{\bar p}^{\widetilde\tau},
\end{equation}
where ${\bar \kappa}(\widetilde\tau,k):=\delta(k)-\bigl(\delta(k)+M(k)\bigr) \exp\left(-\frac{k-\widetilde\tau}{C_0}\right)>0$. 
\begin{theorem}
\label{T5.1}
Let Assumptions \ref{A2.1} and \ref{AS6.1} hold and let $\bar u\in\mathcal U$ be stationary. Let $\tau_1\le\tau<k\le\sigma_1$, $0<k<\frac1m$, and assume that
\[
M(k)<\delta(k)\left[\exp\left(\frac{k-\tau}{C_0}\right)-1\right].
\]
Let $\bar u_\rho\in\mathcal U$ be stationary for $J_\rho$, with
corresponding states $y_\rho$ and adjoints $p_\rho$, and suppose that
\begin{equation}
\label{E5.3}
\|p_\rho-\bar p\|_{L^\infty(\Omega)}\to0,\quad\|y_\rho-\bar y\|_{L^\infty(\Omega)}\to0\quad\text{as }\rho\to0.
\end{equation}
Then there exist constants $\rho_0>0$ and $\kappa>0$ such that
\[
J_\rho''(\bar u_\rho)[h,h] \ge \kappa \|z_{\rho,h}\|_{L^2(\Omega)}^2
\]
for every $|\rho|<\rho_0$ and every $h\in C_{\bar u_\rho}^{\tau}$.
\end{theorem}
\begin{proof}
The function $t\to\delta(k)\left[\exp\left(\frac{k-t}{C_0}\right)-1\right]$ 
is continuous and we can choose $\eta>0$ such that $\tau+\eta<k$ and
\begin{equation*}
M(k)<\delta(k)\left[\exp\left(\frac{k-\tau-\eta}{C_0}\right)-1\right].
\end{equation*}
Choosing $\widetilde\tau=\tau+\eta$ in \eqref{E5.1}, we obtain
\begin{equation}
\label{E5.4}
J''(\bar u)[h,h]\ge\kappa_\eta\|z_{\bar u,h}\|_{L^2(\Omega)}^2\quad
\forall h\in K_{\bar p}^{\tau+\eta},
\end{equation}
where $\kappa_\eta:=\delta(k)-\bigl(\delta(k)+M(k)\bigr)\exp\left(-\frac{k-\tau-\eta}{C_0}\right)>0$. 

By \eqref{E5.3}, after reducing the size of the perturbation if
necessary, we have
\begin{equation}
\label{E5.5}
\|p_\rho-\bar p\|_{L^\infty(\Omega)}<\eta.
\end{equation}

Let $h\in C_{\bar u_\rho}^{\tau}$. We claim that
\begin{equation}
\label{E5.6}
h\in K_{\bar p}^{\tau+\eta}.
\end{equation}
Given $x\in\Omega$ such that $|\bar p(x)|>\tau+\eta$, \eqref{E5.5} implies
\[
|p_\rho(x)|\ge|\bar p(x)|-|p_\rho(x)-\bar p(x)|>\tau+\eta-\eta=\tau.
\]
Since $h\in C_{\bar u_\rho}^{\tau}$, we have $h=0$ a.e. on $[|p_\rho|>\tau]$. It follows that $h=0$ a.e. on $[|\bar p|>\tau+\eta]$, which proves \eqref{E5.6}. Abbreviating $\bar z:=z_{\bar u,h}$, $z_\rho:=z_{\rho,h}$, the functions $\bar z,z_\rho\in H_0^1(\Omega)$ solve
\[
-\operatorname{div}(A\nabla\bar z)+b\bar z=h, \text{ and } -\operatorname{div}(A\nabla z_\rho)+b_\rho z_\rho=h,
\]
where $b:=f_y(\cdot,\bar y)$. Subtracting the two equations gives
\begin{equation}
\label{E5.7}
-\operatorname{div}
\bigl(A\nabla(z_\rho-\bar z)\bigr)+b_\rho(z_\rho-\bar z)=(b-b_\rho)\bar z.
\end{equation}

We take the Poincar\'e constant $C_\Omega$ of $\Omega$ and define $C_R:=\frac{C_\Omega^2}{\lambda_A}$. Testing \eqref{E5.7} with $z_\rho-\bar z$, using $b_\rho\ge0$, the ellipticity of $A$, and the Poincar\'e inequality, gives
\begin{equation}
\label{E5.8}
\|z_\rho-\bar z\|_{L^2(\Omega)}\le a_\rho \|\bar z\|_{L^2(\Omega)},
\end{equation}
for $a_\rho:=C_R \|b_\rho-b\|_{L^\infty(\Omega)}$. Then,
\begin{equation}
\label{E5.9}
\|z_\rho\|_{L^2(\Omega)}\le(1+a_\rho)\|\bar z\|_{L^2(\Omega)}.
\end{equation}
We define $d_\rho:=\|c_\rho-c\|_{L^\infty(\Omega)}$. Then
\begin{align}
&\left|J_\rho''(\bar u_\rho)[h,h]-J''(\bar u)[h,h]\right|
=\left|\int_\Omega\left(c_\rho z_\rho^2-c\bar z^2\right)\,dx\right|
\nonumber\\
&\quad\le d_\rho\|z_\rho\|_{L^2(\Omega)}^2+\|c\|_{L^\infty(\Omega)}
\|z_\rho-\bar z\|_{L^2(\Omega)}\left(\|z_\rho\|_{L^2(\Omega)}+\|\bar z\|_{L^2(\Omega)}
\right).
\label{E5.10}
\end{align}
Combining \eqref{E5.8} and \eqref{E5.9} with
\eqref{E5.10} gives
\begin{equation}
\label{E5.11}
\left|
J_\rho''(\bar u_\rho)[h,h]-J''(\bar u)[h,h]\right|\le\omega_\rho\|\bar z\|_{L^2(\Omega)}^2,
\end{equation}
for $\omega_\rho:=d_\rho(1+a_\rho)^2+\|c\|_{L^\infty(\Omega)}a_\rho(2+a_\rho)$. By \eqref{E5.3} and the assumptions on $f_y$ and $f_{yy}$,
 $\|b_\rho-b\|_{L^\infty(\Omega)}\rightarrow0$ and $\|c_\rho-c\|_{L^\infty(\Omega)}
\rightarrow0$. Therefore, $a_\rho\rightarrow0$, $d_\rho\rightarrow0$, $\omega_\rho\rightarrow0$, as $\rho\to0$. By \eqref{E5.6} and \eqref{E5.4},
\[
J''(\bar u)[h,h]
\ge
\kappa_\eta
\|\bar z\|_{L^2(\Omega)}^2
\]
for every $h\in C_{\bar u_\rho}^{\tau}$. Combining this with \eqref{E5.11} gives
\[
J_\rho''(\bar u_\rho)[h,h]\ge(\kappa_\eta-\omega_\rho)\|\bar z\|_{L^2(\Omega)}^2.
\]
On the other hand, \eqref{E5.9} implies
\[
\|\bar z\|_{L^2(\Omega)}^2\ge\frac{1}{(1+a_\rho)^2}\|z_\rho\|_{L^2(\Omega)}^2.
\]
Thus
\[
J_\rho''(\bar u_\rho)[h,h]
\ge\frac{\kappa_\eta-\omega_\rho}{(1+a_\rho)^2}\|z_\rho\|_{L^2(\Omega)}^2.
\]

Since $a_\rho\to0$ and $\omega_\rho\to0$, there exists $\rho_0>0$ such that $\omega_\rho\le\frac{\kappa_\eta}{2}$ and $(1+a_\rho)^2
\le2$ for every $|\rho|<\rho_0$. Therefore, $\frac{\kappa_\eta-\omega_\rho}{(1+a_\rho)^2}\ge\frac{\kappa_\eta}{4}=:\kappa>0$, and for all $|\rho|<\rho_0$,
\[
J_\rho''(\bar u_\rho)[h,h]\ge\kappa\|z_{\rho,h}\|_{L^2(\Omega)}^2\quad\forall h\in C_{\bar u_\rho}^{\tau}.
\]
\end{proof}

\begin{corollary}
Let $y_{d,\rho}\to y_d$ in $L^r(\Omega)$ and let $\bar u_\rho\in\mathcal U$ be stationary for $J_\rho$ such that $\bar u_\rho\rightharpoonup\bar u$ weakly in $L^2(\Omega)$ as $\rho\to0$. Suppose that the reference control $\bar u$ satisfies
the assumptions of Theorem \ref{T5.1}. Then there exist constants $\rho_0>0$ and $\kappa>0$ such that for $|\rho|<\rho_0$,
\[
J_\rho''(\bar u_\rho)[h,h]\ge\kappa\|z_{\rho,h}\|_{L^2(\Omega)}^2
\quad\forall h\in C_{\bar u_\rho}^{\tau}.
\]
\end{corollary}
\begin{proof}
The weak convergence of the controls and
Theorem \ref{TH2.1} gives $y_\rho=y_{\bar u_\rho}\to\bar y$ in $C(\overline{\Omega})$ and in $H_0^1(\Omega)$. Then $f_y(\cdot,y_\rho)\rightarrow f_y(\cdot,\bar y)$ in $L^\infty(\Omega)$. Subtracting the adjoint equations and the elliptic regularity estimates gives $p_\rho\to\bar p$ in $W^{2,r}(\Omega)$. Since $r>n$, the Sobolev-Morrey embedding gives $p_\rho\rightarrow\bar p$ in $C^1(\overline{\Omega})$. The uniform continuity assumption on $f_{yy}$ on bounded sets, together with $y_\rho\to\bar y$ in $ C(\overline{\Omega})$, gives $f_{yy}(\cdot,y_\rho)\to f_{yy}(\cdot,\bar y)$ in $L^\infty(\Omega)$. Since $p_\rho\to\bar p$ in $L^\infty(\Omega)$, we obtain
\[
c_\rho=1-p_\rho f_{yy}(\cdot,y_\rho)\rightarrow1-\bar p f_{yy}(\cdot,\bar y)=c\quad\text{in }L^\infty(\Omega).
\]
\end{proof}

\appendix
\section{Auxiliary results}
For the main result of this paper we introduce facts that are mostly taken out of \cite{Necas2012}. 

We recall that the boundary of $\Omega$ is said to be Lipschitz if the properties below are satisfied (see, e.g., \cite[Chapter 1.1.3]{Necas2012}):
\begin{enumerate}
    \item There exist constants $\alpha>0$ and $\beta>0$ and local charts $(x_{r1},...,x_{rn})=(x'_r,x_{rn})$, $r=1,2,...,M$ and Lipschitz functions $a_r$ defined on $\overline{\Delta_r}:=\{ x'\in \mathbb R^{n-1} \vert \ \vert x'_{ri}\vert \leq \alpha,~ i=1,2,...,n-1\}$ such that for any $x\in \partial \Omega$, there is a chart in which $x=(x'_r, a_r(x'_r))$.
    \item The elements $(x'_r,x_{rn})$ with $\vert x_{ri}\vert \leq \alpha$, $i=1,2,...,n-1$ and $a_r(x')<x_{rn}<a_r(x'_r)+\beta$ are in $\Omega$ while the ones with $a_r(x')-\beta<x_{rn}<a_r(x'_r)$ are outside of $\overline{\Omega}$.
\end{enumerate}

\begin{definition}
Throughout the appendix, $\Omega\subset\mathbb R^n$ is a bounded
Lipschitz domain, if a smoother boundary is not explicitly assumed. We consider $$Lv:=-\operatorname{div}(A\nabla v)+cv,$$ where $A\in C^{0,1}(\overline\Omega;\mathbb R^{n\times n})$ is symmetric and satisfies
\[
\lambda|\xi|^2\le A(x)\xi\cdot\xi\le\Lambda|\xi|^2
\quad x\in\overline\Omega,\ \xi\in\mathbb R^n.
\]
Further, $c\in L^\infty(\Omega)$ and $c\ge0$ almost everywhere. By $G: L^2(\Omega)\times W^{1,2}(\partial\Omega)\to
W^{1,2}(\Omega)$ we denote the solution map for $Lv=f$ in $\Omega$, $v=g$ on $\partial\Omega$ and define $K:=\|G\|_{\mathcal L(L^2(\Omega)\times W^{1,2}(\partial\Omega), W^{1,2}(\Omega))}$, where the product space is equipped with the norm
\[
\|(f,g)\|:=\left(\|f\|_{L^2(\Omega)}^2+\|g\|_{W^{1,2}(\partial\Omega)}^2\right)^{1/2}.
\]
Further, we write $L_A:=\operatorname*{ess\ sup}_{x\in\Omega}\left(\sum_{k=1}^n \|\partial_k A(x)\|_{\mathrm{op}}^2\right)^{1/2}$.
\end{definition}

The following lemma is taken from \cite[Lemma 1.1, Chapter 5.1.2]{Necas2012}. Below, we sketch the proof to make the dependence of the constant in the estimate on the parameters clear. 
\begin{lemma}[{\cite[Lemma 1.1, Chapter 5.1.2]{Necas2012}}]\label{Neca1}
Let $\Omega$ be a Lipschitz domain and $g\in W^{1,2}(\partial \Omega)$. Then there exists a linear operator extending $g$ to $\Omega$, $g\in W^{1,2}(\Omega)$, such that
\begin{equation}\label{APEQ1}
\| g\|_{W^{1,2}(\Omega)}\leq C(\beta, \sup_{x\in \Omega} \sharp \{ r\vert\ x\in \operatorname{spt} \varphi_r\},\max_r \|\varphi_r\|_{C^1}) \|g\|_{W^{1,2}(\partial \Omega)},
\end{equation}
\end{lemma}
In the proof below, we find that for the constant in \eqref{APEQ1} we have
\begin{align*}
C(\beta, \sup_{x\in \Omega} \sharp \{ r\vert\ x\in \operatorname{spt} \varphi_r\},&\max_r \|\varphi_r\|_{C^1})\\
&\leq\sqrt{\beta \sup_{x\in \Omega} \sharp \{ r\vert\ x\in \operatorname{spt} \varphi_r\} \max\{2,1+2\max_r\|\nabla \varphi_r\|_{L^\infty}^2\}}.
\end{align*}
\begin{proof}
We consider a system of $m\in \mathbb N$ Lipschitz charts such that 
\[
\partial\Omega\cap U_r=\left\{(x_r',x_{rn})\vert \ x_{rn}=a_r(x_r'),\ x_r'\in\Delta_r\right\}, \ r=1,\dots,m
\]
and $U_r\cap\Omega \subset\left\{(x_r',x_{rn})\vert \ x_r'\in\Delta_r,\ a_r(x_r')<x_{rn}<a_r(x_r')+\beta\right\}$. For $g\in W^{1,2}(\partial\Omega)$, we define $\widehat g_r(x_r'):=g(x_r',a_r(x_r'))$, $x_r'\in\Delta_r$, then
\begin{equation}
\|g\|_{W^{1,2}(\partial\Omega)}^2=\sum_{r=1}^m\|\widehat g_r\|_{W^{1,2}(\Delta_r)}^2.
\label{APEQ2}
\end{equation}
All constants involving $W^{1,2}(\partial\Omega)$ below refer
to this choice of norm. We choose nonnegative functions $\varphi_0\in C_c^\infty(\Omega)$, $\varphi_r\in C_c^\infty(U_r),\quad r=1,\ldots,m$, such that $\varphi_0+\sum_{r=1}^m\varphi_r=1$ on $\overline\Omega$. Then,
\[
0\le\sum_{r=1}^m\varphi_r\le1 \quad\text{on }\overline\Omega,\quad\sum_{r=1}^m\varphi_r=1\quad\text{on }\partial\Omega.
\]
We define
\[
g_r(x_r',x_{rn}):=\widehat g_r(x_r')\varphi_r(x_r',x_{rn}),\quad Eg:=\sum_{r=1}^m g_r,
\]
where the $g_r$ is extended by zero outside its chart. Then the trace of $Eg$ equals $g$. We further define $M_0:=\max_{1\leq r\leq m}\|\varphi_r\|_{L^\infty(U_r)}$, and $M_1:= \max_{1\leq r\leq m}\|\nabla\varphi_r\|_{L^\infty(U_r)}$ and since $0\leq\varphi_r\leq1$, it holds $M_0\leq1$. We set $$\bar M :=\operatorname*{sup}_{x\in\Omega}\#\left\{r\in\{1,\dots,m\}\vert \ x\in\operatorname{spt}\varphi_r\right\}\in\{1,..., m\}$$ and at each $x\in\Omega$ at most $\bar M$ of the functions $g_r(x)$ are nonzero. We select a $r$, then
\begin{equation}
\|g_r\|_{L^2(\Omega\cap U_r)}^2\leq \int_{\Delta_r}\int_{a_r(x')}^{a_r(x')+\beta}|\widehat g_r(x')|^2|\varphi_r(x',x_n)|^2\ dx_n\ dx'\leq \beta M_0^2 \|\widehat g_r\|_{L^2(\Delta_r)}^2.
\label{APEQ3}
\end{equation}
It is clear that $\partial_i g_r=\varphi\ \partial_i\widehat g_r+\widehat g_r\ \partial_i\varphi_r$ for $i=1,\dots,n-1$. We further estimate $|\nabla' g_r|^2\leq 2M_0^2|\nabla'\widehat g_r|^2+2|\widehat g_r|^2|\nabla'\varphi_r|^2$. Since $\widehat g_r$ does not depend on $x_n$, $\partial_n g_r=\widehat g_r\ \partial_n\varphi_r$, and $|\partial_n g_r|^2\leq|\widehat g_r|^2|\partial_n\varphi_r|^2$. Therefore,
\begin{align*}
|\nabla g_r|^2&=|\nabla' g_r|^2+|\partial_n g_r|^2\leq 2M_0^2|\nabla'\widehat g_r|^2+|\widehat g_r|^2 \left(2|\nabla'\varphi_r|^2+|\partial_n\varphi_r|^2\right)\\&\leq 2M_0^2|\nabla'\widehat g_r|^2+2M_1^2|\widehat g_r|^2.
\end{align*}
Integrating gives
\begin{equation}
\|\nabla g_r\|_{L^2(\Omega\cap U_r)}^2\leq 2\beta M_0^2 \|\nabla'\widehat g_r\|_{L^2(\Delta_r)}^2 + 2\beta M_1^2 \|\widehat g_r\|_{L^2(\Delta_r)}^2.
\label{APEQ4}
\end{equation}
Then \eqref{APEQ3} and \eqref{APEQ4} imply
\begin{align*}
\|g_r\|_{W^{1,2}(\Omega)}^2&\leq
2\beta M_0^2 \|\nabla'\widehat g_r\|_{L^2(\Delta_r)}^2+ \beta(M_0^2+2M_1^2) \|\widehat g_r\|_{L^2(\Delta_r)}^2.
\end{align*}
Thus, if $C_\varphi:=\max\left\{2M_0^2,\ M_0^2+2M_1^2\right\}$, we have $\|g_r\|_{W^{1,2}(\Omega)}^2\leq\beta C_\varphi  \|\widehat g_r\|_{W^{1,2}(\Delta_r)}^2$. Since $\operatorname{supp}\nabla g_r\subseteq \operatorname{supp}\varphi_r$, we have 
\[ 
\left| \sum_{r=1}^m g_r(x) \right|^2 \leq \bar M \sum_{r=1}^m |g_r(x)|^2 \quad \text{and}\quad \ \left| \sum_{r=1}^m \nabla g_r(x) \right|^2\leq \bar M \sum_{r=1}^m |\nabla g_r(x)|^2.
\]
Therefore 
\begin{align*}
\|Eg\|_{W^{1,2}(\Omega)}^2&\leq \bar M \sum_{r=1}^m \|g_r\|_{W^{1,2}(\Omega)}^2\leq \bar M\beta C_\varphi \sum_{r=1}^m \|\widehat g_r\|_{W^{1,2}(\Delta_r)}^2.
\end{align*}
Using \eqref{APEQ2}, we conclude $\|Eg\|_{W^{1,2}(\Omega)}^2\leq \bar M\beta C_\varphi \|g\|_{W^{1,2}(\partial\Omega)}^2$. Since $M_0\leq1$, the claim follows, for $C=\left[ \bar M\beta\max\left\{2M_0^2,\ M_0^2+2M_1^2\right\}\right]^{1/2}\leq\left[\bar M\beta \max\left\{2,\ 1+2M_1^2\right\}\right]^{1/2}$.
\end{proof}

\begin{lemma}[Version of {\cite[Lemma 1.3, Chapter 5.1.2]{Necas2012}}]\label{Neca3}
Let $\Omega$ be a smooth bounded domain. Let the charts and the partition of unity be as in Lemma \ref{Neca1}. Then, for $f\in L^2(\Omega)$ and $g\in C^\infty(\partial\Omega)$, the solution $v=G(f,g)$ satisfies
\begin{equation}
\|\nabla v\|_{L^2(\partial\Omega)}\leq C_{\partial \Omega} \left(\|f\|_{L^2(\Omega)}^2+\|g\|_{W^{1,2}(\partial\Omega)}^2\right)^{1/2}, \label{APEQ5}
\end{equation}
for $C_{\partial \Omega}^2\leq \frac{2R}{\delta\lambda}+\frac{16\Lambda^2 (1+\bar a^2)^{1/2}}{\delta^2\lambda^2}$ with $\delta:=\frac{1}{\sqrt{1+\bar a^2}}$ and 
\begin{align}
R&:=\left[(n+2)\Lambda  \bar M M_1+L_A\right]K^2+ 2K\left(1+\|c\|_{L^\infty(\Omega)}K\right). \label{APEQ6}
\end{align}
and 
\begin{equation}
\left\|\frac{\partial v}{\partial\nu_A}\right\|_{L^2(\partial\Omega)}
\leq \Lambda C_{\partial \Omega} \left(\|f\|_{L^2(\Omega)}^2+\|g\|_{W^{1,2}(\partial\Omega)}^2\right)^{1/2}. \label{APEQ7}
\end{equation}
\end{lemma}
\begin{proof}
We follow the proof of \cite[Lemma 1.3, Chapter 5.1.2]{Necas2012}. Since $\partial\Omega$ is smooth, $A$ is Lipschitz, $f\in L^2(\Omega)$, and $g\in C^\infty(\partial\Omega)$, elliptic regularity gives $v\in H^2(\Omega)$.
Thus, the traces of $\nabla v$ belong to $L^2(\partial\Omega)$. We set
\[
\bar a:=\max_{1\leq r\leq m}\|\nabla a_r\|_{L^\infty(\Delta_r)},\quad M_1:=\max_{1\leq r\leq m}\|\nabla\varphi_r\|_{L^\infty}.
\]
For each chart, $e_{r,n}$ denotes the $n$-th coordinate vector, such that $-e_{r,n}$ points outward on $\Omega$. We define $h_{r}:=-\varphi_r e_{r,n}$, $h:=\sum_{r=1}^m h_{r}$, where $\varphi_r$ is the partition of unity for the respective chart. Then $\|h\|_{L^\infty(\Omega)}\leq 1$ and at every point at most $\bar M$ of the functions $\varphi_r$ are nonzero, and therefore $\|Dh\|_{L^\infty(\Omega)} \leq  \bar M M_1$. On the $r$-th chart, $h_{r}\cdot\nu= \frac{\varphi_r} {\sqrt{1+|\nabla a_r|^2}}$. Since $\sum_r\varphi_r=1$ on $\partial\Omega$, it follows that $h\cdot\nu \geq\frac{1}{\sqrt{1+\bar a^2}}\ \text{on }\partial\Omega$. Since $A$ is symmetric, we obtain
\begin{equation}
\begin{aligned}
&\partial_k \left[\left(h_k a_{ij}-h_i a_{kj}-h_j a_{ik}\right)\partial_i v \partial_j v \right]=B_{ij} \partial_i v \partial_j v+2h_i \partial_i v\left[-\partial_k(a_{kj}\partial_j v)\right],
\end{aligned}
\label{APEQ8}
\end{equation}
where $B=(\operatorname{div}h)A-(Dh)A-A(Dh)^\top+h_k\partial_k A$. Then
\begin{align}
\|B\|_{L^\infty(\Omega;\mathrm{op})}  &\leq (n+2)\Lambda  \|Dh\|_{L^\infty(\Omega)}+\|h\|_{L^\infty(\Omega)}L_A
\leq(n+2)\Lambda \bar M M_1+L_A. \label{APEQ9}
\end{align}
Since $Lv=f$, we have $-\operatorname{div} (A\nabla v)=f-cv$. Integrating \eqref{APEQ8} over $\Omega$ and applying the divergence theorem gives
\begin{equation*}
\begin{aligned}
I&:=\int_{\partial\Omega}\left(h_k a_{ij}-h_i a_{kj}-h_j a_{ik}\right) \partial_i v \partial_j v\nu_k\ dS \\
&=\int_\Omega B_{ij} \partial_i v \partial_j v\ dx+2\int_\Omega (h\cdot\nabla v)(f-cv)\ dx.
\end{aligned}
\end{equation*}
We have the existence of a constant $K$ with
\begin{equation}
\|v\|_{W^{1,2}(\Omega)}\leq K \left(\|f\|_{L^2(\Omega)}^2+\|g\|_{W^{1,2}(\partial\Omega)}^2\right)^{1/2}.\label{APEQ10}
\end{equation}
From \eqref{APEQ9}, and
\eqref{APEQ10}, we conclude
\begin{align*}
|I|&\leq \bigl[(n+2)\Lambda \bar M M_1+L_A\bigr] \|\nabla v\|_{L^2(\Omega)}^2 \nonumber\\
    &+2\|\nabla v\|_{L^2(\Omega)}\left(\|f\|_{L^2(\Omega)}+\|c\|_{L^\infty(\Omega)}\|v\|_{L^2(\Omega)}\right)\\
    &\leq  R \left(\|f\|_{L^2(\Omega)}^2+\|g\|_{W^{1,2}(\partial\Omega)}^2\right),
\end{align*}
where $R$ is given by \eqref{APEQ6}. We have
\begin{align}
-I&=\int_{\partial\Omega}\Big[(h_i a_{jk}-h_k a_{ji})+(h_j a_{ki}-h_k a_{ji})+h_k a_{ji} \Big] \partial_i v \partial_j v\nu_k\ dS\nonumber\label{eq:I-decomp-Neca3}
\end{align}
We define $j$ vectors with entries $\tau_i^{j}:=\sum_{k=1}^n(h_i a_{jk}-h_k a_{ji})\nu_k$. Then $\tau^{j}\cdot\nu=0$, so $\tau^{j}$ is tangent to $\partial\Omega$ and $|\tau^{j}| \leq 2\Lambda |h|$. Then estimating the three terms $I_i$, $i\in \{1,2,3\}$, we find
\begin{equation}
\begin{aligned}
\frac{1}{\sqrt{1+\bar a^2}}\lambda \|\nabla v\|_{L^2(\partial\Omega)}^2 \leq |I|+ 4\Lambda \|\nabla_\tau v\|_{L^2(\partial\Omega)}\|\nabla v\|_{L^2(\partial\Omega)}.
\end{aligned}
\label{APEQ11}
\end{equation}
Since $v=g$ on $\partial\Omega$, $\nabla_\tau v=\nabla_\tau g$. On the graph $x_n=a_r(x')$, the surface element satisfies $dS=\sqrt{1+|\nabla a_r|^2}\ dx'$, and the metric of the graph is $G_r=I+\nabla a_r\otimes\nabla a_r$. Thus $|\nabla_\tau g|^2=\nabla\widehat g_r^\top G_r^{-1}\nabla\widehat g_r \leq |\nabla\widehat g_r|^2$, 
\begin{align*}
\|\nabla_\tau g\|_{L^2(\partial\Omega)}^2&\leq \sqrt{1+\bar a^2} \sum_{r=1}^m \|\nabla\widehat g_r\|_{L^2(\Delta_r)}^2\leq \sqrt{1+\bar a^2} \|g\|_{W^{1,2}(\partial\Omega)}^2.
\end{align*}
Therefore, 
\begin{equation}
\|\nabla_\tau g\|_{L^2(\partial\Omega)}\leq  (1+\bar a^2)^{1/4} \|g\|_{W^{1,2}(\partial\Omega)}. \label{APEQ12}
\end{equation}
With \eqref{APEQ11}, \eqref{APEQ12}, and  $\|g\|_{W^{1,2}(\partial\Omega)}\leq \left(\|f\|_{L^2(\Omega)}^2+\|g\|_{W^{1,2}(\partial\Omega)}^2\right)^{1/2}$, gives
\begin{equation}
\delta \lambda \|\nabla v\|_{L^2(\partial\Omega)}^2 \leq RF^2+ 4\Lambda (1+\bar a^2)^{1/4} \left(\|f\|_{L^2(\Omega)}^2+\|g\|_{W^{1,2}(\partial\Omega)}^2\right)^{1/2}\|\nabla v\|_{L^2(\partial\Omega)}. \label{APEQ13}
\end{equation}
Young's inequality gives 
\begin{align*}
&4\Lambda (1+\bar a^2)^{1/4} \left(\|f\|_{L^2(\Omega)}^2+\|g\|_{W^{1,2}(\partial\Omega)}^2\right)^{1/2} \|\nabla v\|_{L^2(\partial\Omega)}\\
&\leq \frac{\delta \lambda}{2}\|\nabla v\|_{L^2(\partial\Omega)}^2+ \frac{8\Lambda^2 (1+\bar a^2)^{1/2}} {\delta \lambda}\left(\|f\|_{L^2(\Omega)}^2+\|g\|_{W^{1,2}(\partial\Omega)}^2\right)
\end{align*}
Substituting into \eqref{APEQ13}, we find
\[
\|\nabla v\|_{L^2(\partial\Omega)}^2\leq \left(\frac{2R}{\delta \lambda}+ \frac{16\Lambda^2 (1+\bar a^2)^{1/2}}{\delta ^2\lambda^2}\right)\left(\|f\|_{L^2(\Omega)}^2+\|g\|_{W^{1,2}(\partial\Omega)}^2\right).
\]
Remembering that $\left|\frac{\partial v}{\partial\nu_A}\right|=|A\nabla v\cdot\nu|\leq\Lambda|\nabla v|$, we obtain 
\[
\left\| \frac{\partial v}{\partial\nu_A}\right\|_{L^2(\partial\Omega)}\leq\Lambda C_{\partial \Omega}\left(\|f\|_{L^2(\Omega)}^2+\|g\|_{W^{1,2}(\partial\Omega)}^2\right)^{1/2}.
\]
We define $C_\nu:=\Lambda C_{\partial\Omega}$. The conormal derivative is extended to $g\in W^{1,2}(\partial\Omega)$, by considering $g_k\in C^\infty(\partial\Omega)$ with $g_k\to g$ in $W^{1,2}(\partial\Omega)$. Then By linearity and \eqref{APEQ7},
\[
\left\|\frac{\partial G(f,g_k)}{\partial\nu_A}-\frac{\partial G(f,g_\ell)}{\partial\nu_A}\right\|_{L^2(\partial\Omega)}\leq C_\nu \|g_k-g_\ell\|_{W^{1,2}(\partial\Omega)},
\]
thus the conormal derivatives form a Cauchy sequence in
$L^2(\partial\Omega)$ and its limit defines the unique bounded extension.
\end{proof}

\begin{lemma}[Version of
{\cite[Lemma 1.4, Chapter 5.1.2]{Necas2012}}]
\label{Neca4}
Let $\Omega\subset\mathbb R^N$ be a bounded Lipschitz domain and let $\{\Omega_k\}_{k\in\mathbb N}$ be smooth subdomains of $\Omega$ such that every compact subset of $\Omega$ is contained in a $\Omega_k$ for all sufficiently large $k$ and such that the boundaries $\partial\Omega_k$ are described in the same system of charts as $\partial\Omega$, 
\[
\partial\Omega_k\cap U_r = \left\{(x_r',x_{rN})\; \middle|\ x_{rN}=a_{r,k}(x_r'),\ x_r'\in\Delta_r \right\}, \  r=1,\dots,m.
\]
Assume, for every $r=1,\ldots,m$, that
\[
a_{r,k}\to a_r \quad\text{uniformly on }\overline{\Delta_r} \quad\text{and strongly in }W^{1,2}(\Delta_r),
\]
with $\Omega_k$ lying above the graph of $a_{r,k}$ in each chart and suppose that
\begin{equation*}
\bar a:=\sup_{k\in\mathbb N} \max_{1\leq r\leq m}\|\nabla a_{r,k}\|_{L^\infty(\Delta_r)} <\infty.
\end{equation*}
Let the partition
$\varphi_0,\varphi_1,\ldots,\varphi_m$ from Lemma \ref{Neca3} be selected 
independently of $k$, and define $\bar M$ and $M_1$ as above. Since $\operatorname{supp}\varphi_0$ is a compact subset of $\Omega$, it is contained in $\Omega_k$ for all sufficiently large $k$. After excluding finitely many terms, we thus have
\[
\sum_{r=1}^m\varphi_r=1 \quad\text{on }\partial\Omega_k \quad\text{for every }k.
\]
Let $c\in L^\infty(\Omega)$, $c\geq0$. Let $C_P$ be the Poincar\'e constant for $\Omega$. For $f\in L^2(\Omega)$, let $w_k\in W_0^{1,2}(\Omega_k)$ be the unique weak solution of
\begin{equation}
Lw_k=f|_{\Omega_k} \quad\text{in }\Omega_k, \quad  w_k=0 \quad\text{on }\partial\Omega_k. \label{APEQ14}
\end{equation}
Then for $K_0:=\frac{C_P\sqrt{1+C_P^2}}{\lambda}$,  
\begin{equation}
\|w_k\|_{W^{1,2}(\Omega_k)} \leq K_0\|f\|_{L^2(\Omega_k)}\leq K_0\|f\|_{L^2(\Omega)}, \label{APEQ15}
\end{equation}
The conormal derivatives $t_k := \frac{\partial w_k}{\partial\nu_A} = A\nabla w_k\cdot\nu_k$ satisfy
\begin{equation}
\|t_k\|_{L^2(\partial\Omega_k)}\leq C_{\nu,0}\|f\|_{L^2(\Omega)},
\label{APEQ16}
\end{equation}
where $C_{\nu,0}$ can be chosen as $C_{\nu,0} =\Lambda \left( \frac{R_0}{\delta\lambda} \right)^{1/2}$ with $\delta := \frac{1}{\sqrt{1+\bar a^2}}$ and
\[
R_0:=\left[(N+2)\Lambda \bar M M_1+L_A \right]K_0^2 +2K_0(1+\|c\|_{L^\infty(\Omega)}K_0).
\]
After extending $w_k$ by zero to $\Omega$, $w_k\to w $ strongly in $W_0^{1,2}(\Omega)$, for $w$, the unique weak solution of $Lw=f$ in $\Omega$, $w=0$ on $\partial\Omega$. After passing to a subsequence, for every chart $r=1,\dots,m$ there exists $t_r\in L^2(\Delta_r)$ such that $t_k(x_r',a_{r,k}(x_r')) \rightharpoonup t_r(x_r')$ weakly in $L^2(\Delta_r)$.
\end{lemma}
\begin{proof}
We take $u\in W_0^{1,2}(\Omega_k)$ and extend $u$ by zero to $\Omega$ which is again denoted by $u\in W_0^{1,2}(\Omega)$. From this it is clear that 
\begin{equation}
\|u\|_{L^2(\Omega_k)}\leq C_P\|\nabla u\|_{L^2(\Omega_k)}.
\label{APEQ17}
\end{equation}
Testing \eqref{APEQ14} with $w_k$ and using $c\geq0$ gives
\begin{align*}
\lambda \|\nabla w_k\|_{L^2(\Omega_k)}^2
&\leq \int_{\Omega_k} A\nabla w_k\cdot\nabla w_k + c|w_k|^2\ dx = \int_{\Omega_k}fw_k\ dx \\
&\leq \|f\|_{L^2(\Omega_k)} \|w_k\|_{L^2(\Omega_k)} \leq C_P \|f\|_{L^2(\Omega_k)} \|\nabla w_k\|_{L^2(\Omega_k)},
\end{align*}
and $\|\nabla w_k\|_{L^2(\Omega_k)} \leq \frac{C_P}{\lambda} \|f\|_{L^2(\Omega_k)}$ and, by \eqref{APEQ17}, $\|w_k\|_{L^2(\Omega_k)} \leq \frac{C_P^2}{\lambda} \|f\|_{L^2(\Omega_k)}$. Thus
\begin{align*}
\|w_k\|_{W^{1,2}(\Omega_k)}^2 &= \|w_k\|_{L^2(\Omega_k)}^2 + \|\nabla w_k\|_{L^2(\Omega_k)}^2 \leq \frac{C_P^2(1+C_P^2)}{\lambda^2} \|f\|_{L^2(\Omega_k)}^2,
\end{align*}
which proves \eqref{APEQ15}. If $G_k: L^2(\Omega_k) \to W_0^{1,2}(\Omega_k)$ denotes the homogeneous Dirichlet solution operator, then $\sup_{k\in\mathbb N} \|G_k\|_{\mathcal L(L^2(\Omega_k),W_0^{1,2}(\Omega_k))} \leq K_0$. We apply Lemma \ref{Neca3} on each smooth domain $\Omega_k$. All bounds are independent of $k$, since the coefficients are restrictions of the coefficients on $\Omega$. The geometric constants are also uniform, $\sup_{k,r} \|\nabla a_{r,k}\|_{L^\infty(\Delta_r)} \leq \bar a$, so the vector fields used in the proof of Lemma \ref{Neca3} satisfy $h_k\cdot\nu_k \geq \frac{1}{\sqrt{1+\bar a^2}}$ on $\partial\Omega_k$ and $\|h_k\|_{L^\infty}\leq1$, $\|Dh_k\|_{L^\infty} \leq \bar M M_1$. Therefore the coefficient occurring in the identity is bounded uniformly by $(N+2)\Lambda \bar M M_1+L_A$. Since $w_k=0$ on $\partial\Omega_k$, its tangential derivative vanishes, $\nabla_\tau w_k=0$ on $\partial\Omega_k$. Thus the tangential boundary term appearing in the proof of Lemma \ref{Neca3} is zero. If $X_k := \|\nabla w_k\|_{L^2(\partial\Omega_k)}$, Rellich gives $\delta\lambda X_k^2\leq R_0 \|f\|_{L^2(\Omega_k)}^2$, where, using \eqref{APEQ15}, $R_0=\left[ (N+2)\Lambda \bar M M_1+L_A \right]K_0^2 + 2K_0(1+\|c\|_{L^\infty(\Omega)}K_0)$, 
\begin{equation*}
\|\nabla w_k\|_{L^2(\partial\Omega_k)} \leq \left( \frac{R_0}{\delta\lambda} \right)^{1/2} \|f\|_{L^2(\Omega_k)}.
\end{equation*}
Since $|A\nabla w_k\cdot\nu_k| \leq \Lambda|\nabla w_k|$, we obtain
\[
\left\| \frac{\partial w_k}{\partial\nu_A} \right\|_{L^2(\partial\Omega_k)} \leq
\Lambda \left( \frac{R_0}{\delta\lambda} \right)^{1/2} \|f\|_{L^2(\Omega_k)},
\]
which proves \eqref{APEQ16}. We extend every $w_k$ by zero from $\Omega_k$ to $\Omega$.
The sequence is bounded in $W_0^{1,2}(\Omega)$. After passing to a subsequence, $w_k\rightharpoonup \widetilde w$ weakly in $W_0^{1,2}(\Omega)$. Let $\phi\in C_0^\infty(\Omega)$. By the assumption, $\operatorname{spt}\phi\subset\Omega_k$ for all sufficiently large $k$. Thus
\[
\int_{\Omega}\left(A\nabla w_k\cdot\nabla\phi + cw_k\phi \right)\ dx = \int_{\Omega}f\phi\ dx
\]
for all sufficiently large $k$. Passing to the limit gives
\[
\int_{\Omega}\left( A\nabla\widetilde w\cdot\nabla\phi + c\widetilde w\phi \right)\ dx = \int_{\Omega}f\phi\ dx.
\]
Thus $\widetilde w$ is the weak solution of $L\widetilde w=f$, $\widetilde w\in W_0^{1,2}(\Omega)$. By uniqueness, $\widetilde w=w$. Since $w_k$ solves \eqref{APEQ14},
\begin{equation*}
\int_{\Omega}
\left(A\nabla w_k\cdot\nabla w_k + cw_k^2 \right)\ dx =\int_{\Omega}fw_k\ dx.
\end{equation*}
By weak convergence in $W_0^{1,2}(\Omega)$, $\int_{\Omega}fw_k\ dx \rightarrow\int_{\Omega}fw\ dx$. Testing the equation for $w$ with $w$ gives $\int_{\Omega}fw\ dx =\int_{\Omega}\left(A\nabla w\cdot\nabla w+cw^2 \right)\ dx$. Therefore
\begin{equation*}
\lim_{k\to\infty}\int_{\Omega}\left(A\nabla w_k\cdot\nabla w_k+cw_k^2\right)\ dx=\int_{\Omega}\left(A\nabla w\cdot\nabla w+cw^2\right)\ dx.
\end{equation*}
Together with the weak convergence and the coercivity of the bilinear form, this implies $w_k\rightarrow w$ strongly in $W_0^{1,2}(\Omega)$. For each fixed $r$ define $T_{r,k}(x_r') := t_k(x_r',a_{r,k}(x_r'))$, $x_r'\in\Delta_r$. Since $dS_k = \sqrt{1+|\nabla a_{r,k}(x_r')|^2}\ dx_r' \geq dx_r'$, we have
\begin{align*}
\|T_{r,k}\|_{L^2(\Delta_r)}^2&\leq \int_{\Delta_r} |T_{r,k}(x_r')|^2 \sqrt{1+|\nabla a_{r,k}(x_r')|^2} \ dx_r' \\
&\leq \|t_k\|_{L^2(\partial\Omega_k)}^2 \leq C_{\nu,0}^2 \|f\|_{L^2(\Omega)}^2.
\end{align*}
Thus $\{T_{r,k}\}_k$ is bounded in $L^2(\Delta_r)$. Since there are only finitely many charts, a diagonal subsequence
can be chosen such that, for every $r=1,\dots,m$, $T_{r,k} \rightharpoonup t_r$ weakly in $L^2(\Delta_r)$.
\end{proof}

\begin{theorem}[{\cite[Theorem 1.1, Chapter 5.1.2]{Necas2012}}]
\label{NecaTheorem1}
Let $\Omega$ be a bounded Lipschitz domain and let $Lw=-\operatorname{div}(A\nabla w)+cw$, under the assumptions of Lemma \ref{Neca4}. For every $f\in L^2(\Omega)$, let $w\in W^{1,2}_0(\Omega)$ be the weak solution of
\[
Lw=f \quad \text{in }\Omega, \quad  w=0 \quad \text{on }\partial\Omega.
\]
Then there exists a unique weak conormal derivative
$\partial_{\nu_A}w\in L^2(\partial\Omega)$ satisfying
\[
\|\partial_{\nu_A}w\|_{L^2(\partial\Omega)} \le C_{\nu,0}\|f\|_{L^2(\Omega)},
\]
where $C_{\nu,0}$ is the constant from the above lemma. Moreover, for $u\in W^{1,2}(\Omega)$,
\begin{equation*}
\int_\Omega uf\ dx = -\int_{\partial\Omega} u\frac{\partial w}{\partial\nu_A}\ dS + \int_\Omega \left( A\nabla u\cdot\nabla w+cuw \right)\ dx.
\end{equation*}
The map $f\mapsto\partial_{\nu_A}w$ is linear and bounded from $L^2(\Omega)$ to $L^2(\partial\Omega)$.
\end{theorem}
\begin{proof}
Let $\Omega_k$ and $w_k$ be as in the above lemma, and consider $t_k:=A\nabla w_k\cdot\nu_k$ on $\partial\Omega_k$. The extensions by zero of $w_k$ converge strongly to $w$ in $H_0^1(\Omega)$, and
\[
\|t_k\|_{L^2(\partial\Omega_k)}\le C_{\nu,0}\|f\|_{L^2(\Omega)}.
\]
We first consider $u\in C^\infty(\mathbb R^n)$. Green's identity on the smooth domain $\Omega_k$ gives
\[
\begin{aligned}
B_k(u)&:=\int_{\Omega_k}\bigl(A\nabla u\cdot\nabla w_k+cuw_k-fu\bigr)\ dx=\int_{\partial\Omega_k}u\,t_k\ dS.
\end{aligned}
\]
Strong convergence of the extensions imply
\[
B_k(u)\to B(u):=\int_\Omega\bigl(A\nabla u\cdot\nabla w+cuw-fu\bigr)\,dx.
\]
We show convergence of the boundary norms. In the $r$th chart, write
\[
\Psi_{r,k}(x')=(x',a_{r,k}(x')), \quad\Psi_r(x')=(x',a_r(x')),
\]
and
\[
J_{r,k}:=\sqrt{1+|\nabla a_{r,k}|^2},\quad J_r:=\sqrt{1+|\nabla a_r|^2}.
\]
The graph convergence gives
\[
\|J_{r,k}-J_r\|_{L^2(\Delta_r)}\le\|\nabla a_{r,k}-\nabla a_r\|_{L^2(\Delta_r)}
 \to0.
\]
Further, $u\circ\Psi_{r,k}$ and
$\varphi_r\circ\Psi_{r,k}$ converge uniformly to $u\circ\Psi_r$ and $\varphi_r\circ\Psi_r$. Using the partition of unity, we obtain
\[
\begin{aligned}
\|u\|_{L^2(\partial\Omega_k)}^2&=\sum_{r=1}^m\int_{\Delta_r}|u\circ\Psi_{r,k}|^2(\varphi_r\circ\Psi_{r,k})J_{r,k}\ dx'\\
 &\to\sum_{r=1}^m\int_{\Delta_r}|u\circ\Psi_r|^2(\varphi_r\circ\Psi_r)J_r\ dx'=\|u\|_{L^2(\partial\Omega)}^2.
\end{aligned}
\]
Thus, $|B(u)|=\lim_{k\to\infty}|B_k(u)|\le C_{\nu,0}\|f\|_{L^2(\Omega)}\|u\|_{L^2(\partial\Omega)}$ and $B(u)$ depends only on the trace of $u$. By density, the Riesz representation theorem then gives a unique $t\in L^2(\partial\Omega)$ such that
\[
B(u)=\int_{\partial\Omega}u\,t\,dS, \quad \|t\|_{L^2(\partial\Omega)}\le C_{\nu,0}\|f\|_{L^2(\Omega)}.
\]
Again by density the trace operator
$H^1(\Omega)\to L^2(\partial\Omega)$ is continuous. Thus the identity holds for $u\in H^1(\Omega)$. Defining $\partial_{\nu_A}w:=t$ and rearranging gives the Green identity. Uniqueness follows from density, and linearity follows from linearity of the Dirichlet problem and of $B$ in $f$. Let be any chartwise weak limit
$t_k\circ\Psi_{r,k}\rightharpoonup t_r$ be given from the lemma. Testing the identity with smooth functions
supported in the $r$th chart and using $J_{r,k}\to J_r$ in $L^2(\Delta_r)$ gives $t_r=t\circ\Psi_r$ almost everywhere on $\Delta_r$.
\end{proof}

\begin{theorem}[Coarea formula {\cite[Theorem 3.2.12]{zbMATH03280855}}]\label{coarea}
Let $f:\mathbb{R}^n \to \mathbb{R}$ be Lipschitz continuous and let
$g:\mathbb{R}^n \to [0,\infty]$ be measurable. Then
\[
\int_{\mathbb{R}^n} g(x) |\nabla f(x)|\ dx =\int_{\mathbb{R}}\left(\int_{f^{-1}(s)}g(x)\ d\mathcal{H}^{n-1}(x)\right)\ ds.
\]
\end{theorem}

\bibliographystyle{siamplain}
\bibliography{sn-bibliography}

\end{document}